\documentclass[10pt]{amsart}

\usepackage{amsfonts,amssymb,amscd,amstext}
\usepackage[a4paper,hmargin=3.5cm,vmargin=4cm]{geometry}
\usepackage{hyperref}
\usepackage{graphicx}
\usepackage{color}

\renewcommand{\leq}{\leqslant}
\renewcommand{\geq}{\geqslant}

\newcommand{\ptl}{\partial}

\newcommand{\rr}{{\mathbb{R}}}

\newcommand{\hh}{{\mathbb{H}}}
\newcommand{\nn}{{\mathbb{N}}}

\newcommand{\hhh}{\mathcal{H}}

\newcommand{\escpr}[1]{\big<#1\big>}
\newcommand{\Sg}{\Sigma}
\newcommand{\Om}{\Omega}
\newcommand{\eps}{\varepsilon}

\newcommand{\ga}{\gamma}
\newcommand{\Ga}{\Gamma}
\newcommand{\dg}{\dot{\ga}}

\newcommand{\n}{\nabla}
\newcommand{\m}{\int\limits }

\newcommand{\s}{\Sigma}

\newcommand{\Si}{\Sigma}

\newcommand{\mc}[1]{\mathcal{#1}}

\newcommand{\phM}{$(M,g_\hhh,\omega,J)$ }
\newcommand{\phMM}{pseudo-hermitian 3-manifold $(M,g_\hhh,\omega,J)$ }

\newcommand{\conMM}{contact sub-Riemannian 3-manifold $(M,g_\hhh,\omega)$ }

\DeclareMathOperator{\divv}{div}
\DeclareMathOperator{\tor}{Tor}

\DeclareMathOperator{\IA}{g(\tau(Z),\nu_h)}

\newtheorem{theorem}{Theorem}[section]
\newtheorem{proposition}[theorem]{Proposition}
\newtheorem{lemma}[theorem]{Lemma}
\newtheorem{corollary}[theorem]{Corollary}

\theoremstyle{definition}

\newtheorem{remark}[theorem]{Remark}

\newtheorem{definition}{Definition} 

\theoremstyle{remark}

\numberwithin{equation}{section}

\begin{document}

\title[First and second variation formulae for the sub-riemannian area]{First and second variation formulae for the sub-Riemannian area in three-dimensional pseudo-hermitian manifolds}

\author[M. Galli]{Matteo Galli} \address{Departamento de
Geometr\'{\i}a y Topolog\'{\i}a \\
Universidad de Granada \\ E--18071 Granada \\ Espa\~na}
\email{galli@ugr.es}

\date{\today}

\thanks{Research supported by  MCyT-Feder grant MTM2010-21206-C02-01 and J. A. grant P09-FMQ-5088}
\subjclass[2000]{53C17, 49Q20} 
\keywords{sub-Riemannian geometry, minimal surfaces, stable surfaces, variation formulas, pseudo-hermitian manifolds, roto-traslation group}

\begin{abstract}
We calculate the first and the second variation formula for the sub-Riemannian area in three dimensional pseudo-hermitian manifolds. We consider general variations that can move the singular set of a $\mathcal{C}^2$ surface and non-singular variation for $\mc{C}^2_\hhh$ surfaces. These formulas enable us to construct a stability operator for non-singular $\mc{C}^2$ surfaces and another one for $\mc{C}^2$ (eventually singular) surfaces.  Then we can obtain a necessary condition for the stability of a non-singular surface in a pseudo-hermitian 3-manifold in term of the pseudo-hermitian torsion and the Webster scalar curvature. Finally we classify complete stable surfaces in the roto-traslation group $\mc{RT}$. 
\end{abstract}

\maketitle

\thispagestyle{empty}

\bibliographystyle{amsplain}

\tableofcontents

\section{Introduction}

Over the last years considerable efforts have been devoted to the study of critical points of the area functional in sub-Riemannian geometry and a large number of research papers have provided variation formulas in sub-Riemannian manifolds, \cite{Ba-SC-Vi}, \cite{ChengHwang2nd}, \cite{CJHMY}, \cite{DGNnotable}, \cite{DGNAV}, \cite{Da-Ga-Nh-Pa},  \cite{Hl-Pa2}, \cite{Hl-Pa}, \cite{Ri-Ro-Hu}, \cite{Ro-Hu}, \cite{Mo}, \cite{Ri-Ro}, \cite{Ro} among others.  
In \cite{Ri-Ro} the authors were able to produce a first variation formula for $\mc{C}^2$ surfaces and solve the isoperimetric problem in this class. This result gave a partial positive answer to a celebrated conjecture of Pansu, \cite{P}, which states that isoperimetric regions in $\hh^1$ are a one-parameter family of topological balls which are not metric balls. In \cite{Ri-Ro-Hu} the second variation formula allows to classify the entire stationary graphs in the Heisenberg group $\hh^1$ and to solve the sub-Riemannian Bernstein problem in the class of $\mc{C}^2$ horizontal graphs. These techniques were generalized for the pseudo-hermitian 3-sphere and 3-Sasakian sub-Riemannian space forms, \cite{Ro-Hu} and  \cite{Ro}. 
Also \cite{DGNnotable} and \cite{Da-Ga-Nh-Pa} are two interesting works related to the sub-Riemannian Bernstein problem in $\hh^1$.  In \cite{DGNnotable} the authors construct a family of complete area-stationary intrinsic graphs that are not stable and in \cite{Da-Ga-Nh-Pa} is shown that $C^2$ complete stable area-stationary Euclidean graphs with empty singular set must be vertical planes. Other remarkable work is \cite{Ba-SC-Vi}, where the first and the second variation formulas for intrinsic graphs are used to give a description of the horizontal entire minimal intrinsic graphs in $\hh^1$ and to show that the only stable ones are vertical planes. 
In \cite{CJHMY} and in \cite{Mo} the authors find general first and second variation formulas for surfaces inside pseudo-hermitian 3-manifolds and Carnot groups respectively. Finally in \cite{ChengHwang2nd} the authors present a definition of sub-Riemannian structure in terms of a metric, perhaps degenerate, defined on the cotangent bundle. In this way they can be able to unify the notions of area and mean curvature in Riemannian, pseudo-hermitian and contact sub-Riemannian geometries and to give a first variation formula  for an hypersurface in such a sub-Riemannian manifold.

The aim of this paper is generalize the first and the second variation formulas to general pseudo-hermitian manifolds in the spirit of \cite{Ri-Ro-Hu}, \cite{Ro-Hu}, \cite{Ri-Ro} and \cite{Ro}. We stress that with respect to the cited works we introduce some technical improvements. We use the pseudo-hermitian connection and the horizontal Jacobian, Lemma \ref{jacobianoorizzontale}, that allow us to simplify considerably some proofs, to obtains formulas with geometric terms more adapted to the pseudo-hermitian structure and to move surfaces  of class $\mc{C}^2_\hhh$ outside the singular set. We remark that the presence of a non-vanishing pseudo-hermitian torsion generates some non-trivial problems in the computation of the second variation formula and in its applications.

The work is organized as follows. 
In section two we introduce some notations and preliminaries. 

In section three we produce a first variation formula for $\mc{C}^2$ (eventually singular) surfaces. 

In section four we study special vector fields along characteristic curves, which are the generalization to Jacobi vector fields in Sasakian manifolds \cite[Section~3]{Ri-Ro} and \cite[Section~3]{Ro}. These vector fields will play a key role to prove instability results in section nine. 

In section five we study the local behavior of the singular set combining results of \cite{CJHMY} and \cite{Ga-Ri}. Moreover in Theorem \ref{genusbound} we prove that the genus of a closed, bounded mean curvature surface immersed in a contact sub-Riemannian manifold is less or equal then one.
Using the characterization given in \cite{Ga-Ri} of how the mean curvature change applying the Darboux's diffeomorphism, we can extent the same result proved in \cite{CJHMY} for  pseudo-hermitian manifolds.

In section six we produce a first variation formula for non-singular $\mc{C}^2_\hhh$ surfaces. We underline the interest to work with  $\mc{C}^2_\hhh$ surfaces, that are only $\mc{C}^1$ from the Euclidean point of view.  In sub-Riemannian manifolds there are examples of minimizers with low regularity, \cite{CHY} and \cite{Ri1}. 

In section seven we present  second variation formulas in the regular set $\s-\s_0$ or in a neighborhood of the singular set $\s_0$. These formulas allow us to construct two different stability operators in section eight. The first one is for $\mc{C}^2$ non-singular surfaces and coincides with the second variation formula in \cite{CJHMY} for $\mc{C}^3$ surfaces. The second stability operator we construct is for $\mc{C}^2$ (eventually singular) surfaces. Recently in \cite{ChengHwang2nd} the authors produce a second variational formula moving a singular lines, but we underline that they need $\mc{C}^3$ surfaces since in the proof they differentiate the mean curvature.  

In section nine we prove a necessary condition for $\mc{C}^2$ stable minimal surfaces with empty singular set in a large class of pseudo-hermitian manifolds, that included the uni-modular Lie groups, Proposition \ref{necessariastabilita}
\begin{quote}
\emph{Let $\Sigma$ be a $\mathcal{C}^2$ complete orientable surface with empty singular set immersed in a \phMM. We suppose that $g(R(T,Z)\nu_h,Z)-Z(g(\tau(Z),\nu_h))=0$ on $\s$ and the quantity $W-c_1g(\tau(Z),\nu_h)$ is constant along characteristic curves. We also assume that all characteristic curves in $\Sigma$ are either closed or non-closed. If $\Sigma$ is a stable minimal surface, then $W-c_1g(\tau(Z),\nu_h)\leq 0$ on $\s$. Moreover, if $W-c_1g(\tau(Z),\nu_h)=0$ then $\Sigma$ is a stable vertical surface.}
\end{quote}
This is an important class since in \cite{Pe} is shown that simply connected contact Riemannian 3-manifolds homogeneous in the sense of Bootby and Wang, \cite{MR0112160} (there exists a connected Lie group acting transitively as a group of contact diffeomorphisms), are Lie groups. The condition that we found involves the Webster scalar curvature $W$ and the pseudo-hermitian torsion $\tau$ of the manifold that are pseudo-hermitian invariants. 

Finally in section ten we apply previous results to the roto-traslation group $\mc{RT}$. We give a classification of area-stationary surfaces with non-empty singular set, Lemma \ref{stazionariopuntosingolare} and Lemma \ref{stazionariaconcurva}

\begin{quote}
\emph{Let $\s$ be a complete area-stationary surface of class $\mc{C}^2$  with non-empty singular set. Then $\s$ is a right-handed helicoid or a plane $\{(x,y,\theta)\in\mc{RT}:ax+by+c=0, a,b\in\rr,c\in\mathbb{S}^1\}$.}
\end{quote}

and the classification of complete stable surfaces, Theorem \ref{classificationstablesRT}
\begin{quote}\emph{Let $\s$ be a $\mc{C}^2$ stable, immersed, oriented and complete surface in $\mc{RT}$. Then}
\begin{enumerate}
\item \emph{if $\s$ is a non-singular surface, then it is a vertical plane;}
\item  \emph{if $\s$ is a surface with singular set, then it is a right-handed helicoid.}
\end{enumerate}
\end{quote}
The $\mc{RT}$ group is interesting for two reasons. From the geometric point of view it is one of the simplest pseudo-hermitian manifolds which does not have zero torsion.  Moreover it is a model of the visual cortex of the human eye which play an important role in the theory of image reconstruction, \cite{MR2235475} and \cite{MR2366414}. Given a boundary curve $\Ga$, we can reconstruct an image by solving a Plateau's problem. This is equivalent to find a stable minimal surface $\s$ with boundary $\Ga$, i.e. to find $\s$ such that $A'(\s)(0)=0$ and $A''(\s)(0)\geq 0$, for variations that fix $\partial\s=\Ga$.

\section{Preliminaries}

A three-dimensional contact manifold \cite{Bl} is a three-dimensional smooth manifold $M$ so that there exists a one-form $\omega$ such that $d\omega$ is non-degenerate when restricted to $\mathcal{H}:=ker(\omega)$. Since
\[
d\omega(X,Y)=X(\omega(Y))-Y(\omega(X))-\omega([X,Y])
\]
the \emph{horizontal distribution} $\mathcal{H}$ is completely non-integrable. It is well known the existence of a unique \emph{Reeb vector field} T in $M$ so that
\begin{equation}\label{def:reeb}
\omega(T)=1, \hspace{1cm} (\mathcal{L}_T \omega)(X)=0,
\end{equation}
where $\mathcal{L}$ is the Lie derivative and $X$ any smooth vector field on $M$. It is a direct consequence that $\omega\wedge d\omega$ is an orientation form on $M$. A well-known example of contact manifold is the Euclidean space $\rr^{3}$ with the contact one-form
\begin{equation}\label{1formacanonica}
\omega_0:=dt+xdy-ydx.
\end{equation}

A \emph{contact transformation} between contact manifolds is a diffeomorphism preserving the horizontal distribution. A \emph{strict contact transformation} is a diffeomorphism preserving the contact one-form.  A strict contact transformation preserves the Reeb vector field. Darboux's Theorem \cite[Theorem~3.1]{Bl} shows that, given a contact manifold $M$ and some point $p\in M$, there exists an open neighborhood $U$ of $p$ and a strict contact transformation $f$ from $U$ into a open set of $\rr^3$ with its standard contact structure induced by $\omega_0$. Such a local chart will be called a \emph{Darboux chart}. 

A positive definite metric $g_\mathcal{H}$ on $\mathcal{H}$ defines a \emph{contact sub-Riemannian} manifold $(M,g_{\mathcal{H}},\omega)$ on $M$ \cite{Mong}. The first Heisenberg group is the contact sub-Riemannian manifold $\hh^1\equiv(\rr^3,g_0,\omega_0)$, where $g_0$ is the Riemannian metric on $\hhh$ defined requiring that 
\[
X=\frac{\partial}{\partial x}+y\frac{\partial}{\partial t}, \hspace{1cm} Y=\frac{\partial}{\partial x}+y\frac{\partial}{\partial t}
\] 
form an orthonormal basis at each point.

The length of a piecewise horizontal curve $\ga:I\to M$ is~defined by
\[
L(\ga):=\int_I |\ga'(t)|\,dt,
\]
where the modulus is computed with respect to the metric $g_\mathcal{H}$. The Carnot-Carath\'eo\-do\-ry distance $d(p,q)$ between $p$, $q\in M$ is defined as the infimum of the lengths of piecewise smooth horizontal curves joining $p$ and $q$. A minimizing geodesic is any curve $\ga:I\to M$ such that $d(\ga(t),\ga(t'))=|t-t'|$ for each $t$, $t'\in I$.  From \cite[Chap.~5]{Mong} a minimizing geodesic in a contact sub-Riemannian manifold is a smooth curve that satisfies the geodesic equations, i.e., it is normal.

The metric $g_\mathcal{H}$ can be extended to a Riemannian metric $g$ on $M$ by requiring that $T$ be a unit vector orthogonal to $\mathcal{H}$. The scalar product of two vector fields $X$ and $Y$ with respect to the metric $g$ will be often denoted by $\escpr{X,Y}$ instead of $g(X,Y)$. The Levi-Civita connection induced by $g$ will be denoted by $D$. An important property of the metric $g$ is that the integral curves of the Reeb vector field $T$  are \emph{geodesics} \cite[Theorem~4.5]{Bl}. 

A usual class defined in contact geometry is the one of contact Riemannian manifolds, see \cite{Bl}, \cite{Ta}. Given a contact manifold, one can assure the existence of a Riemannian metric $g$ and an $(1,1)$-tensor field $J$ so that
\begin{equation}
\label{eq:contactriem}
g(T,X)=\omega(X), \quad 2g(X,J(Y))=d\omega(X,Y),\quad J^2(X)=-X+\omega(X)\,T.
\end{equation}
The structure given by $(M,\omega,g,J)$ is called a contact Riemannian manifold. The class of contact sub-Riemannian manifolds is different from this one. Recall that, in our definition, the metric $g_\mathcal{H}$ is given, and it is extended to a Riemannian metric $g$ in $TM$. However, there is not in general an $(1,1)$-tensor field $J$ satisfying all conditions in \eqref{eq:contactriem}. Observe that the second condition in \eqref{eq:contactriem} uniquely defines $J$ on $\mathcal{H}$, but this $J$ does not satisfy in general the third condition in \eqref{eq:contactriem}, as it is easily seen in $(\rr^3,\omega_0)$ choosing an appropriate  positive definite metric in $\text{ker}(\omega_0)$. When $M$ is three-dimensional the structure $(M, \omega, g, J)$ is equivalent to a strongly pseudo-convex pseudo-hermitian structure \cite[Corollary~6.4]{Bl} and we will call briefly \phM \,  a \emph{pseudo-hermitian manifold}. 

The Riemannian volume form $dv_g$ in $(M,g)$ is Popp's measure \cite[\S~10.6]{Mong}. The volume of a set $E\subset M$ with respect to the Riemannian metric $g$ will be denoted by $|E|$. 

A \emph{contact isometry} in $(M,g_\mathcal{H},\omega)$ is a strict contact transformation that preserves $g_\mathcal{H}$. Contact isometries preserve the Reeb vector fields and they are isometries of the Riemannian manifold $(M,g)$.

In a \conMM, we define a linear operator $J:\hhh\rightarrow \hhh$ on an orthonormal basis $\{X,Y\}$ of $\hhh$ respect to the metric $g_\hhh$ by 
\begin{equation}\label{def:J}
g(J(X),Y)=-g(J(Y),X)=sgn(c_1), \quad g(J(X),X)=g(J(Y),Y)=0,
\end{equation}
where  we have denoted $c_1(p)=-g([X,Y](p),T_p)$. We remark that $c_1(p)$ never vanish since $span\{X,Y\}=T M$ and $sgn(c_1)$ equals $1$ or $-1$ in the whole manifold. Furthermore $J$ can be extended to the whole tangent space by requiring $J(T)=0$. Now we define a connection $\n$, that we will call the \emph{(contact) sub-Riemannian connection},  as the unique connection having non-vanishing torsion defined by
\begin{equation}\label{def:torsion}
\tor(X,Y)=g(X,T)\tau(Y)-g(Y,T)\tau(X)+c_1 g(J(X),Y)T,
\end{equation}
where $\tau:TM\rightarrow \hhh$ is defined by 
\[
\tau(V)=-\frac{1}{2}J(\mathcal{L}_T J)(V)
\]
for all $V\in TM$. Clearly $\tau$ vanishes outside $\hhh$. Alternatively if we consider the endomorphism 
\[
\sigma(V):=D_V T:TM\rightarrow \hhh. 
\]
we have that 
\begin{equation}\label{def:sigma}
g(\sigma(V),Z)= g(\tau(V),Z)+\frac{c_1}{2}g(J(V),Z).
\end{equation}
Last equation can be viewed as an alternative definition of $J$ and $\tau$, where $J$ and $\tau$ are antisymmetric and symmetric respectively.
We shall call $\tau$ the \emph{(contact) sub-Riemannian torsion}. We remark that $\n$ and $\tau$ are generalization of the well-known pseudo-hermitian connection and pseudo-hermitian torsion in \phMM,\cite[Appendix]{CJHMY} and \cite{Dr-To}. From the above definitions follows easily
\begin{equation}\label{nablaT}
\nabla_V T=0,
\end{equation}
\begin{equation}\label{nablaJ}
(\nabla_V J)Z=0
\end{equation}
and
\begin{equation}\label{Jrotazione}
g(J(V),V)=0,
\end{equation}
for all $V,Z\in TM$. Here $J^2=-Id$ on $\hhh$ but satisfies the second equation in \eqref{eq:contactriem} if and only if $(M,g,J)$ is a pseudo-hermitian manifold. It implies the normalization $c_1=2$ and at my knowledge all definition of pseudo-hermitian manifolds and Riemannian contact manifolds require it implicitly. But exists interesting examples that do not satisfy $c_1=2$ as the roto-traslation group $RT$ that we will study in the last section.
The difference between Levi-Civita and pseudo-hermitian connections can be computed using the Koszul's formulas as in \cite[p.38]{Dr-To}
\begin{equation}\label{phvsLC}
2g(D_{X}Y-\n_{X}Y,Z)=g(Tor(X,Z),Y)+g(Tor(Y,Z),X)-g(Tor(X,Y),Z).
\end{equation}

In a \conMM  we can generalize the definition of the \emph{Webster scalar curvature} known in \phMM    by
\begin{equation}\label{def:W}
W:=-g(R(X,Y)Y,X), 
\end{equation}
where $\{X,Y\}$ is an orthonormal basis of $\hhh$ and $R$ is the pseudo-hermitian curvature tensor defined by
\begin{equation}
R(Z,W)V=\n_W \n_Z V-\n_Z\n_W V+\n_{[Z,W]}V,
\end{equation}
for all $Z,W,V\in TM$.

In the following we restrict ourself to the case in which $c_1$ is a constant. We briefly call a such manifold a pseudo-hermitian 3-manifold $(M,g_\hhh,\omega,J)$, as it has analogous properties respect to a pseudo-hermitian manifold defined in \cite[Appendix]{CJHMY} and \cite{Dr-To}.

\section{The first variation formula for $\mc{C}^2$ surfaces.}

We define the \emph{area} of a $\mathcal{C}^1$ surface $\Si$ immersed in $M$ by
\begin{equation}\label{eq:area}
A(\Si)=\int\limits_{\Si} |N_h| d\Si,
\end{equation}
where $N$ is the unit normal vector with respect to the metric $g$, $N_h$ is the orthogonal projection of $N$ to $\mathcal{H}$  and $d\Si$ is the Riemannian area element of $\Si$. The singular set $\Si_0$ consists of those points $p$ where $\hhh_p$ coincides with the tangent plane $T_p\Si$ of $M$.  We define the \emph{horizontal unit normal vector} $\nu_h(p)$ and the \emph{characteristic vector field} $Z(p)$ by
\begin{equation}\label{def:Znu_h}
\nu_h(p):=\frac{N_h(p)}{|N_h(p)|}, \quad Z(p):=J(\nu_h)(p)
\end{equation}
for all $p\in\Si-\Si_0$. Since $Z_p$ is orthogonal to $\nu_h$ and horizontal, we get that $Z_p$ is tangent to $\s$ and generates $T_p\s\cap\hhh_p$. We call \emph{characteristic curves} of $\s$ the integral curves of $Z$ in $\s-\s_0$. Now setting
\begin{equation}\label{def:S}
S:=g(N,T)\nu_{h}-|N_{h}|T
\end{equation}
we get that $\{Z_p,S_p\}$ is an orthonormal basis of $T_{p}\s$ for $p\in\s -\s_{0}$. 

Now we consider a $\mathcal{C}^1$ vector field $U$ with compact support on $\Si$ and denote by $\Si_t$ the variation of $\Si$ induced by $U$, i.e., $\Si_t=\{exp_p(tU_p)|p\in\Si\}$, where $exp$ is the exponential map of $M$ with respect to $g$. Furthermore we denote by $B$ the Riemannian shape operator and by $\theta$ the 1-form associated to the connection $\n$ and $\nu_h$
 \begin{equation}\label{def:theta}
 \theta(v):=g(\n_{v}\nu_{h},Z),
 \end{equation}
for all $v\in T_p M$.

\begin{lemma}\label{conti} Let $\s$ be an oriented immersed $\mathcal{C}^2$ surface in a contact sub-Riemannian three-dimensional manifold $(M,g_\hhh,\omega)$.  Consider a point $p\in\s -\s_{0}$, the horizontal Gauss map $\nu_{h}$ and the basis $\{Z,S\}$ of $T_{p}M$ already defined. For any $v\in T_{p}M$ we have
\begin{enumerate}
\item[(i)] $|N_{h}|Z(|N_{h}|)=-g(N,T)Z(g(N,T))$;
\item[(ii)] $|N_{h}|^{-1}Z(g(N,T))=|N_{h}|Z(g(N,T))-g(N,T)Z(|N_{h}|)$;
\item[(iii)] $g(B(Z),S)=\frac{c_1}{2}-g(\tau(Z),\nu_h)+|N_{h}|^{-1}Z(g(N,T))=-g(\sigma(Z),\nu_h)+|N_h|^{-1}Z(g(N,T))$;
\item[(iv)] $g(B(S),Z)=-g(N,T)^2g(\tau(\nu_h),Z)+\frac{c_1}{2}(|N_h|^2-g(N,T)^2)-|N_h|\theta(S)$;
\item[(v)] $|N_{h}|^{-1}Z(g(N,T))=-c_1g(N,T)^2+|N_h|^2g(\tau(Z),\nu_h)-|N_h|\theta(S)$.
\end{enumerate}
\end{lemma}

\begin{proof} From $Z(|N_{h}|^2)=Z(1-g(N,T)^2)$ we immediately obtain (i). Using (i) and $|N|=1$ we get
\begin{align*}
|N_{h}|Z(g(N,T))-g(N,T)Z(|N_{h}|)&=(|N_h|+|N_h|^{-1}g(N,T))Z(g(N,T))\\
&=|N_h|^{-1}Z(g(N,T))
\end{align*}
which proves (ii). From $N=g(N,T)T+|N_h|\nu_h$, \eqref{def:sigma} and \eqref{def:S} we have
\begin{align*}
g(D_{Z}N,S)&=Z(g(N,T))g(T,S)+g(N,T)g(\sigma(Z),S)+Z(|N_h|)g(\nu_h,S)+|N_h|g(D_Z \nu_h,S)\\
&=g((c_1/2)J(Z)+\tau(Z),\nu_h)+|N_{h}|Z(g(N,T))-g(N,T)Z(|N_{h}|),
\end{align*}
where we have used 
\[
g(D_Z \nu_h,S)=-|N_h|g(D_Z \nu_h,T)=|N_h|g(\sigma(Z),\nu_h).
\]
Now from (ii) we get (iii). On the other hand 
\begin{align*}
g(D_{S}N,Z)&=g(N,T)g(\sigma(S),Z)+|N_h|g(D_S \nu_h,Z),
\end{align*}
by \eqref{phvsLC} and \eqref{def:S} we obtain (iv).
Finally we get (v) subtracting (iii) and (iv) since $g(B(Z),S)-g(B(S),Z)=0$. 
\end{proof}

The next lemma is proved in \cite{Rireach} for the Heisenberg group $\mathbb{H}^n$, but it holds in a general \conMM with the same proof.

\begin{lemma}\label{jacobianoorizzontale}  Let $\s$ be a $\mc{C}^1$ surface in $M$, $p\in\s$ and $\{E_1,E_2\}$ any basis of $T_{p}\s$. Then we have
\begin{equation*}
|N_{h}|(p)=\frac{|v_{p}|}{G(E_1,E_2)^{1/2}},
\end{equation*}
where $v_{p}:=g(T_{p},E_1)E_2-g(T_{p},E_2)E_1$ and $G(E_1,E_2)$ is the Gram determinant of $\{E_1,E_2\}$.
\end{lemma}

\begin{proof} We consider 
\begin{equation}\label{lemmaiacor}
N_{h}=\lambda E_1+\mu E_2+|N_{h}|^2 N
\end{equation}
so that $|N_{h}|^2g(N,T) =-(\lambda g(E_1,T)+\mu g(E_2,T))$. From $N_{h}=N-g(N,T)T$ we have 
\begin{equation}\label{nonso}
g(N_{h},E_{i})=-g(N,T)g(T,E_{i}),i=1,2.
\end{equation}  
Now compute $\lambda$ and $\mu$ taking scalar product in \eqref{lemmaiacor} with $E_1$ and $E_2$ in (\ref{lemmaiacor}) and using \eqref{nonso}  we have

\begin{displaymath} 
\left( \begin{array}{ccc} 
\lambda \\
\mu 
\end{array} \right) =
\frac{-g(N,T)}{G(E_1,E_2)} \left( \begin{array}{ccc} 
g(E_2,E_2) & -g(E_1,E_2) \\
-g(E_1,E_2) & g(E_1,E_1) 
\end{array} \right) 
\left( \begin{array}{ccc} 
g(T,E_1)\\
g(T,E_2)
\end{array} \right).
\end{displaymath} 
Hence we have obtained 
\begin{equation*}
|N_{h}|^2 g(N,T)=\frac{g(N,T)}{G(E_1,E_2)}|v|^2
\end{equation*}
which prove the statement  in the case $g(N,T)\neq 0$. If $g(N,T)=0$ we simply check that $|v|^2=G(E_1,E_2)^2$, writing $E_1$ and $E_2$ in term of an orthonormal basis $\{w,T\}$ of $T_{p}\s$.
\end{proof}

Now we introduce the notion of intrinsic regularity, \cite{FSSC}, \cite{FSSC2} and \cite{FSSC3}. Let $\Omega$ be an open subset of  $M$, we say $f:\Omega \rightarrow \mathbb{R}$ of class $\mathcal{C}^1_{\hhh}$ in $\Omega$ when $Xf$ exists and it is continuous for any $X\in \hhh$. We define $f\in \mathcal{C}^k_{\hhh}(\Omega)$ when $Xf\in\mathcal{C}^{k-1}_{\hhh}(\Omega)$ for all $X\in \hhh$. Since $c_1$ is  a real constant immediately we obtain that $f\in\mathcal{C}^{2k}_{\hhh}(\Omega)$ implies $f\in\mathcal{C}^{k}(\Omega)$.  
We define a surface $\s$ a \emph{$\hhh$-regular surface} of class $\mc{C}^k_\hhh$ if for any $p\in \s$ exist $B_r(p)$, a metric ball of radius $r$ centered in $p$, and a function $f\in \mc{C}^k_\hhh$ such that
\[ 
\s\cap B_r(p)=\{p\in B_r(p): f(p)=0, \n_\hhh f(p)\neq 0\}, 
\]
see \cite{FSSC} for the definition in the Heisenberg group.

\begin{lemma}\label{devergenceteorem}
Let $\s$ be an oriented immersed $\mathcal{C}^2$ surface in a contact sub-Riemannian three-dimensional manifold $(M,g_\hhh,\omega)$  and let $f\in\mathcal{C}^1 (M)$. Then we have
\begin{equation*}
div_{\s}(fS)=S(f)+fg(N,T)\theta(Z)-f|N_{h}|g(\tau(Z),Z),
\end{equation*}
and
\begin{equation*}
\begin{aligned}
div_{\s}(f Z) =Z(f)-fg(N,T)\theta(S)+fg(N,T)|N_{h}|g(\tau(\nu_h),Z)+c_1fg(N,T)|N_h|g(J(\nu_h),Z),
\end{aligned}
\end{equation*} 
where $div_\Sigma$ is the Riemannian divergence with respect to an orthonormal basis of $T\Sigma$.

\end{lemma}

\begin{proof} We have 
\[div_{\s}(fZ)=Z(f)+fg(D_{S}Z,S)
\] 
and by \eqref{def:S}
\begin{equation*}
g(D_{S}Z,S)=g(N,T)g(D_{S}Z,\nu_{h})-g(N,T)|N_{h}|g(D_{\nu_{h}}Z,T)
\end{equation*}
and using \eqref{phvsLC} we prove the second equation. For the first one we note that
\[
div_{\s}(fS)=S(f)+fg(D_{Z}S,Z)
\] 
and we can conclude using
\[
g(D_{Z}S,Z)=g(N,T)g(D_{Z}\nu_{h},Z)-|N_{h}|g(D_{Z}T,Z)
\] 
together with \ref{phvsLC}.
\end{proof}

Now we can present the key Lemma to obtain the first variation.

\begin{lemma}
Let $\s$ be an oriented immersed $\mc{C}^1$ surface in a contact sub-Riemannian three-dimensional manifold $(M,g_\hhh,\omega)$. Then the first variation of the sub-Riemannian area induced by the  vector field $U$, that is $\mathcal{C}^1_{0}(\s)$ along the variation, is given by
\begin{equation}\label{var}
\begin{aligned}
\frac{d}{ds}\Big{|}_{s=0} A(\varphi_{s}(\s))=\m_{\s-\s_0}\{ & -S(g(U,T))+c_1g(N,T)g(J(\nu_h),U_h)\\
&+|N_{h}|g(\n_{Z}U_h,Z)+|N_{h}|g(U,T)g(\tau(Z),Z)  \}d\s.
\end{aligned}
\end{equation}

\end{lemma}

\begin{proof} For every $p\in\s$ and the orthonormal basis $\{Z,S\}$ of $T_{p}\s$, we consider extensions $E_{1}(s),E_{2}(s)$ of  $Z,S$ along the curve $s\mapsto \varphi_{s}(p)$ so that
$[E_{i},U]=0$, i.e., the vector fields $E_{i}$ are invariant under the flow generated by $U$. By Lemma \ref{jacobianoorizzontale} the Jacobian of the map $\varphi_{s}$ at $p$ is given by $G(E_{1}(s),E_{2}(s))^{1/2}$. We get

\begin{equation*}
\begin{aligned}
A(\varphi_{s}(\s))&= \m_{\Sigma} |V(s)|d\Si,
\end{aligned}
\end{equation*}
where $V(s):=g(T,E_{1}(s))E_{2}(s)-g(T,E_{2}(s))E_{1}(s)$. We can express the first derivative of the area as
\begin{equation*}
\frac{d}{ds}\Big{|}_{s=0} A(\varphi_{s}(\s-\s_0))=\m_{\s} \frac{g(\n_{U}V,V(0))}{|V(0)|}d\s.
\end{equation*}
Now $g(T,E_1(0))=0$ and $g(T,E_2(s))=-|N_h|$ imply $V(0)=|N_{h}|Z$. Since $[E_{i},U]=0$ and \eqref{nablaT} we have
\begin{equation*}
\frac{g(D_{U}V,V(0))}{|V(0)|}=-U(g(E_{2},T))+|N_{h}|g(\n_{U}E_{1},Z)=-(g(\n_U E_{2},T))+|N_{h}|g(\n_{U}E_{1},Z)
\end{equation*}
and substituting we obtain
\begin{equation*}
\frac{d}{ds}\Big{|}_{s=0} A(\varphi_{s}(\s))=\m_{\s-\s_0}\{-g(\n_{U}E_{2},T)+|N_{h}|g(\n_{U}E_{1},Z)\}d\s.
\end{equation*}
Finally since
\begin{equation*}
\begin{aligned}
g(\n_U E_{2},T))&=g(\n_{S} U,T)-c_1g(N,T)g(J(\nu_h),U)\\
&=S(g(U,T))-c_1g(N,T)g(J(\nu_h),U)
\end{aligned}
\end{equation*}
and from
\begin{equation*}
\begin{aligned}
g(\n_{U}E_{1},E_{1})&=g(\n_{Z}U_h, Z)+g(U,T)g(\tau(Z),Z)
\end{aligned}
\end{equation*}
we get \eqref{var}.

\end{proof}

Now we are able to get variation formulas in generic directions.


\begin{corollary}\label{1var:tang} Let $\s$ be an oriented immersed $\mathcal{C}^2$ surface in a contact sub-Riemannian three-dimensional manifold $(M,g_\hhh,\omega)$. Then the first variation of the area induced by the tangent vector field  $U=lZ+hS$, with $l,h\in\mathcal{C}^1_{0}(\s-\s_0)$, is 
\begin{equation*}
\frac{d}{ds}\Big{|}_{s=0} A(\varphi_{s}(\s))=\m_{\s} \divv_{\s}(|N_{h}|U))  d\s.
\end{equation*}
Furthermore when $\partial\s=\emptyset$ the above term vanishes.
\end{corollary}

\begin{proof} By \eqref{var} we get
\begin{equation*}
\begin{aligned}
\frac{d}{ds}\Big{|}_{s=0} A(\varphi_{s}(\s))&=\m_{\s-\s_0}\{ c_1lg(N,T)g(J(\nu_h),Z)+|N_h|Z(l)  \}d\s\\
&+\m_{\s-\s_0}\{ S(|N_h|h)+hg(N,T)|N_h|g(\n_{Z}\nu_h,Z)-h|N_h|^2g(\tau(Z),Z)  \}d\s\\
&=\m_{\s}\divv_\s (l|N_h|Z)d\s+\m_{\s}\divv_\s (h|N_h|S)d\s,
\end{aligned}
\end{equation*}
where we have used $|N_h|Z(l)=Z(|N_h|l)-lZ(|N_h|)$, Lemma \ref{conti}, formula \eqref{var} and  Lemma \ref{devergenceteorem}. When $\partial\s=\emptyset$ we can use the Riemannian divergence theorem to prove that the variation vanishes.
\end{proof}

\begin{corollary}\label{1var:normal} Let $\s$ be an oriented immersed $\mathcal{C}^2$ surface in a \conMM. Then the first variation of the area induced by a normal vector field  $U=uN$, with $u\in\mathcal{C}^1_{0}(\s)$, is
\begin{equation*}
\begin{aligned}
\frac{d}{ds}\Big{|}_{s=0} A(\varphi_{s}(\s))&=\m_{\s}u \,g(\n_{Z}\nu_h,Z)d\s-\m_\s \divv_\s (ug(N,T)S)d\s.
\end{aligned}
\end{equation*}
Furthermore if $u\in\mc{C}^1_0(\s-\s_0)$ we get
\[
\frac{d}{ds}\Big{|}_{s=0} A(\varphi_{s}(\s))=\m_{\s}u \,g(\n_{Z}\nu_h,Z)d\s.
\]
\end{corollary}

\begin{proof} By \eqref{var} and Lemma \ref{devergenceteorem} we get
\begin{equation*}
\begin{aligned}
\frac{d}{ds}\Big{|}_{s=0} A(\varphi_{s}(\s))&=\m_{\s}\{ -S(g(N,T)u)+u|N_h|^2g(\n_{Z}\nu_h,Z)+ug(N,T)|N_h|g(\tau(Z),Z)  \}d\s\\
&=\m_{\s}u \, g(\n_{Z}\nu_h,Z)d\s-\m_\s \divv_\s (ug(N,T)S)d\s.
\end{aligned}
\end{equation*}
When $u\in\mc{C}^1_0(\s-\s_0)$, we can use the Riemannian divergence theorem to conclude
\[
\m_\s \divv_\s (ug(N,T)S)d\s=0.
\]
\end{proof}

\begin{remark} When $(M,g_\hhh,\omega)$ is the Heisenberg group $\mathbb{H}^1$ we have that Corollary \ref{1var:normal}  coincides with \cite[Lemma~4.3]{Ri-Ro}. Furthermore Corollary \ref{1var:normal} coincides with \cite[eq.~(2.8')]{CJHMY} where the authors consider non-singular variations in a \phMM. Different versions of Corollary \ref{1var:normal} can be found also in \cite{Mo}, \cite{Mo1} and \cite{Hl-Pa}, \cite{Hl-Pa2}, for Carnot groups and vertically rigid manifolds, respectively. 
\end{remark}

\begin{definition} Let $\s$ be a surface of class $\mathcal{C}^2_{\hhh}$. Corollary \ref{1var:normal} allows us to define the \emph{mean curvature} of $\s$ in a point $p\in\s-\s_{0}$ as
\begin{equation}\label{eq:meancur}
H:=-g(\n_{Z}\nu_{h},Z).
\end{equation}
We call \emph{minimal surface} a surface of class $\mathcal{C}^2_{\hhh}$ whose mean curvature $H$ vanishes.
\end{definition}
We note that our definition of mean curvature coincides with \cite{Am-SC-Vi}, \cite{CJHMY}, \cite{Hl-Pa} and \cite{Ri-Ro} among others, for surfaces of class $\mathcal{C}^2$ and it is motived by Proposition \ref{1var:c2h} in Section \ref{sec:1c2h}. In \cite{Rireach} the author also defines the mean curvature for surfaces of class $\mc{C}^2_\hhh$.

\section{Characteristic Curves and Jacobi-like vector fields.}

In  this section we give a characterization of characteristic curves in a constant mean curvature surface and we define special vector fields along characteristic curves that are the natural generalization of Jacobi vector fields along geodesics in Sasakian sub-Riemannian manifolds. 


\begin{proposition} Let $\s$ be an oriented immersed $\mathcal{C}^2_\hhh$ surface of constant mean curvature $H=c_1\lambda$ in a contact sub-Riemannian three-dimensional manifold $(M,g_\hhh,\omega)$. Then, outside the singular set, the equation of characteristic curves is 
\begin{equation}
\label{eq:carcurvesph}
\n_{Z}Z+c_1\lambda\, J(Z)=0,
\end{equation}
where $\n$ denote the pseudo-hermitian connection. We will call $\lambda$ the curvature of the characteristic curve.
\end{proposition}

\begin{proof}
It is an immediate consequence of  \eqref{nablaT}, \eqref{eq:meancur} and $|Z|=1$.
\end{proof}

\begin{remark} Let $\gamma$ be a Carnot-Caratheodory geodesic in a \phMM. Then the unit tangent vector field $\dot{\gamma}$ to $\gamma$ satisfies \cite[Proposition~15]{Ru}
\begin{equation}
\label{geod}
\begin{cases}
      & \nabla_{\dot{\gamma}}\dot{\gamma}+c_1\lambda J(\dot{\gamma})=0 \\
      & \dot{\gamma}(\lambda)=-\frac{1}{c_1}g(\tau(\dot{\gamma}),\dot{\gamma}),
\end{cases}
\end{equation}
where $\n$ is the pseudo-hermitian connection. This implies that characteristic curves in a constant mean curvature surface are  sub-Riemannian geodesics if and only if $g(\tau(\dot{\gamma}),\dot{\gamma})=0$. For instance, this is satisfied in manifolds with vanishing torsion.

\end{remark}

\begin{proposition}\label{jacobifields} We consider a \phMM, a curve $\alpha:I\rightarrow M$ of class $\mc{C}^1$ defined on some interval $I\subset M$ and a $\mc{C}^1$ unit horizontal vector field $U$ along $\alpha$. For fixed $\lambda\in \rr$, suppose we have a well-defined map $F:I\times I'\rightarrow M$ given by $F(\varepsilon,s)=\gamma_\eps (s)$, where $I'$ is a open interval containing the origin, and $\gamma_\eps (s)$ is a characteristic curve of curvature $\lambda$ with initial conditions $\gamma_\eps (0)=\alpha(\eps)$ and $\dot{\gamma_\eps}(0)=U(\eps)$. Then the vector field $V_\eps(s):=(\partial F/\partial \eps)(\eps,s)$ satisfies the following properties:
\begin{itemize}
\item [(i)] $V_\eps$ is a $\mc{C}^{\infty}$ vector field along $\gamma_\eps$ and satisfies $[\dot{\gamma}_\eps,V_\eps]=0$;
\item [(ii)] along $\gamma_\eps$ we have
\[
\dot{\gamma}_\eps(\lambda g(V_\eps,T)+g(V_\eps,\dot{\gamma}_\eps))=-g(V_\eps,T)g(\tau(\dot{\gamma}_\eps),\dot{\gamma}_\eps),
\]
in particular $\lambda g(V_\eps,T)+g(V_\eps,\dot{\gamma}_\eps)$ is constant along sub-Riemannian geodesics;
\item [(iii)] $V_\eps$ satisfies the equation 
\begin{equation}\label{eq:jacobi}
\begin{aligned}
&V_\eps''+R(\dot{\gamma_\eps},V_\eps)\dot{\gamma}_\eps+c_1\lambda\{J(V_\eps')+g(V_\eps,T)J(\tau(\dot{\gamma}_\eps))\}+\n_{\dot{\gamma}_\eps}Tor(V_\eps,\dot{\gamma}_\eps)=0,
\end{aligned}
\end{equation}
\[
\n_{\dot{\gamma}_\eps}Tor(V_\eps,\dot{\gamma}_\eps)=-g(V_\eps,T)''T+g(V_\eps,T)'\tau(\dot{\gamma}_\eps)+g(V_\eps,T)\n_{\dot{\gamma}_\eps} \tau(\dot{\gamma}_\eps)
\]
where $V'$ denotes the covariant derivative along $\gamma_\eps$  and $R$ the curvature tensor with respect to the pseudo-hermitian connection.
\item [(iv)] the vertical component of $V_\eps$ satisfies the differential equation
\[
g(V_\eps,T)'''+ \beta_1(s) g(V_\eps,T)'+c_1 \beta_2(s) g(V_\eps,T)=0,
\]
with
\[
\beta_1(s)=W-c_1g(\tau(\dot{\gamma}_\eps),J(\dot{\gamma}_\eps))+c_1^2\lambda^2
\]
\[
\beta_2(s)=c_1\lambda g(\tau(\dot{\gamma}_\eps),\dot{\gamma}_\eps)+g(R(\dot{\gamma}_\eps,T)\dot{\gamma}_\eps,J(\dot{\gamma}_\eps))-\dot{\gamma}_\eps(g(\tau(\dot{\gamma}_\eps),J(\dot{\gamma}_\eps))),
\]
where $W$ is the pseudo-hermitian scalar curvature and $'$ is the derivative respect to $s$.
\end{itemize}
\end{proposition}

\begin{proof} For simplicity we avoid the subscript $\eps$ in the computation. The proof of (i) is analogous to the one of \cite[Lemma~3.3~(i)]{Ro}.  From $[\dg,V]=0$ and \eqref{def:torsion} we have
\begin{equation}\label{vt1}
\begin{aligned}
g(V,T)'&=\dot{\gamma}(g(V,T))=g(\n_{\dg} V,T)=g(\n_{V}\dot{\ga}+Tor(\dg,V),T)\\
&=g(Tor(\dot{\gamma},V),T)=c_1g(J(\dot{\gamma}),V).
\end{aligned}
\end{equation}
Here $'$ denotes the derivative of a function. \eqref{vt1} together with
\[
g(V,\dot{\gamma})'=g(Tor(\dot{\gamma},V),\dot{\gamma})-c_1\lambda g(V,J(\dot{\gamma}))=-g(V,T)g(\tau(\dot{\gamma}),\dot{\gamma})-c_1\lambda g(V,J(\dot{\gamma}))
\]
proves (ii). Now using \eqref{nablaJ} we get
\[
\n_V J(\dot{\gamma})=J(\n_V \dot{\gamma})=J(V')+g(V,T)J(\tau(\dot{\gamma})),
\]
that permits us to compute $\n_V(\n_{\dot{\gamma}}\dot{\gamma}+c_1\lambda J(\dot{\gamma}))$ to obtain the first equation in (iii). The second one is simply obtained using \eqref{def:torsion} and \eqref{vt1}.

To prove (iv) we have, differentiating \eqref{vt1}
\begin{equation}\label{vt2}
\frac{1}{c_1}g(V,T)''=c_1\lambda g(\dot{\gamma},V)+g(J(\dot{\gamma}),V')
\end{equation}
and consequently
\[
\begin{aligned}
\frac{1}{c_1}g(V,T)'''&=\dot{\gamma}(c_1\lambda g(\dot{\gamma},V))+c_1\lambda g(\dot{\gamma},V')+g(V'',J(\dot{\gamma}))\\
&=2c_1\lambda g(V,\dot{\gamma})'+c_1\lambda^2g(V,T)'+g(V'',J(\dot{\gamma})).
\end{aligned}
\]
Taking into account (ii) we get
\begin{equation}\label{vt3}
\frac{1}{c_1}g(V,T)'''=-c_1\lambda^2g(V,T)'- 2c_1\lambda g(V,T)g(\tau(\dot{\gamma}),\dot{\gamma})+g(V'',J(\dot{\gamma})).
\end{equation}
The only term you have to dial with is $g(R(\dot{\ga},V)\dot{\ga}+c_1\lambda J(V'),J(V'))$. Now by point (iii) we have
\[
\begin{aligned}
g(V'',J(\dot{\gamma}))=&-g(R(\dot{\gamma},V)\dot{\gamma},J(\dot{\gamma}))-c_1\lambda\{g(V,T)g(\tau(\dot{\gamma}),\dot{\gamma})+g(V',\dot{\gamma})\}\\
&-g(V,T)'g(\tau(\dot{\gamma}),J(\dot{\gamma}))-g(V,T)g(\n_{\dot{\gamma}}\tau(\dot{\gamma}),J(\dot{\gamma})).
\end{aligned}
\]
From $V=g(V,T)T+g(V,\dot{\ga})\dot{\ga}+g(V,J(\dot{\ga}))J(\dot{\ga})$ we obtain
\[
-g(R(\dot{\gamma},V)\dot{\gamma},J(\dot{\gamma}))=-g(V,J(\dot{\gamma}))W-g(V,T)g(R(\dot{\gamma},T)\dot{\gamma},J(\dot{\gamma})).
\]
Furthermore since 
\[
-g(\n_{\dot{\gamma}}\tau(\dot{\gamma}),J(\dot{\gamma}))=-\dot{\gamma}(\tau(\dot{\gamma}),J(\dot{\gamma}))+c_1\lambda g(\tau(\dot{\gamma}),\dot{\gamma})
\]
and
\[
g(V',\dot{\gamma})=g(V,\dot{\gamma})'-g(\n_{\dot{\gamma}}\dot{\gamma},V)=-g(V,T)g(\tau(\dot{\gamma}),\dot{\gamma})
\]
we finally get
\begin{equation}\label{V2J}
\begin{aligned}
g(V'',J(\dot{\gamma}))=&-\frac{1}{c_1}g(V,T)'W+c_1\lambda g(V,T)g(\tau(\dot{\gamma}),\dot{\gamma})-g(\tau(\dot{\gamma}),J(\dot{\gamma}))g(V,T)'\\
&-g(V,T)\{ g(R(\dot{\gamma},T)\dot{\gamma},J(\dot{\gamma}))-\dot{\gamma}(g(\tau(\dot{\gamma})),J(\dot{\gamma}))  \}.
\end{aligned}
\end{equation}
We conclude summing \eqref{vt3} and \eqref{V2J} and simplifying.
\end{proof}

\begin{definition} Let $\ga:I\rightarrow \s$ be a characteristic curve, where $I$ is a real interval and $\s$ is a surface. A vector field $V$ along $\ga$ is called a \emph{Jacobi-like field} if it satisfies \eqref{eq:jacobi} for all $s\in I$.
\end{definition}

\begin{remark} Special cases of Proposition \ref{jacobifields}  can be found in \cite{Ch-Ya}, \cite{Ri-Ro}  and \cite{Ro}. 
\end{remark}

\section{The structure of the singular set.}\label{singularset}

The local model of a three-dimensional contact sub-Riemannian manifold is the contact manifold $(\rr^{3},\omega_0)$, where $\omega_0$ defined in (\ref{1formacanonica}) is the standard contact form in $\rr^{3}$, together with an arbitrary positive definite metric $g_{\mathcal{H}_0}$ in $\mathcal{H}_0$. A basis of the horizontal distribution is given by
\[
X:=\frac{\ptl}{\ptl x}+y\,\frac{\ptl}{\ptl t},\qquad Y:=\frac{\ptl}{\ptl y}+x\,\frac{\ptl}{\ptl t},\qquad 
\]
and the Reeb vector field is
\[
T:=\frac{\ptl}{\ptl t}.
\]
The metric $g_{\mathcal{H}_0}$ will be extended to a Riemannian metric on $\rr^{3}$ so that the Reeb vector field is unitary and orthogonal to $\mathcal{H}_0$. We shall usually denote the set of vector fields $\{X,Y\}$ by $\{Z_1,Z_{2}\}$. The coordinates of $\rr^{3}$ will be denoted by $(x,y,t)$, and the first $2$ coordinates will be abbreviated by $z$. We shall consider the map $F:\rr^{2}\to\rr^{2}$ defined by
\[
F(x,y):=(-y,x).
\]

Given a $C^2$ function $u:\Om\subset\rr^{2}\to\rr$ defined on an open subset $\Om$, we define the graph $G_u:=\{(z,t): z\in\Om, t=u(z)\}$. By \eqref{eq:area}, the sub-Riemannian area of the graph is given by
\[
A(G_u)=\int_{G_u} |N_h|\,dG_u,
\]
where $dG_u$ is the Riemannian metric of the graph and $|N_h|$ is the modulus of the horizontal projection of a unit normal to $G_u$. We consider on $\Om$ the basis of vector~fields $
\big\{\tfrac{\ptl}{\ptl x},\tfrac{\ptl}{\ptl y}\big\}$.

By the Riemannian area formula
\begin{equation}
\label{eq:dgu}
dG_u=\text{Jac}\,d\mathcal{L}^{2},
\end{equation}
where $d\mathcal{L}^{2}$ is Lebesgue measure in $\rr^{2}$ and $\text{Jac}$ is the Jacobian of the canonical map $\Om\to G_u$ given by
\begin{equation}
\label{eq:jac}
\text{Jac}=\{\det(g)+g_{11}(u_y+x)^2+g_{22}(u_x-y)^2-2g_{12}(u_x-y)(u_y+x)\}^{1/2}
\end{equation}
where $g$ is the matrix of the metric, with elements $g_{ij}:=g(Z_i,Z_j)$. 

Let us compute the composition of $|N_h|$ with the map $\Om\to G_u$. The tangent space $T G_u$ is spanned by
\begin{equation}
\label{eq:zi}
X+(u_x-y)\,T, \quad Y+(u_y+x)\,T.
\end{equation}
So the projection to $\Om$ of the singular set $(G_u)_0$ is the set $\Om_0\subset\Om$ defined by $\Om_0:=\{z\in\Om: (\nabla u +F)=0\}$, where $\n$ is the Euclidean gradient in $\rr^2$. Let us compute a \emph{downward pointing} normal vector $\tilde{N}$ to $G_u$ writing
\begin{equation}
\label{eq:tilden}
\tilde{N}=\sum_{i=1}^{2}(a_i Z_i)-T. 
\end{equation}
The horizontal component of $\tilde{N}$ is $\tilde{N}_h=\sum_{i=1}^{2} a_iZ_i$.  We have
\[
\sum_{i=1}^{2} a_ig_{ij}=g(\tilde{N}_h,Z_j)=g(\tilde{N},Z_j)=-(\nabla u+F)_j\escpr{\tilde{N},T}=(\nabla u +F)_j,
\]
since $Z_j$ is horizontal, $\tilde{N}$ is orthogonal to $Z_j$ defined by \eqref{eq:zi}, and \eqref{eq:tilden}. Hence 
\[
(a_1, a_{2})=b(\nabla u+F),
\]
where $b$ is the inverse of the matrix $\{g_{ij}\}_{i,j=1,2}$. So we get
\begin{equation}
\label{eq:modtilden}
|\tilde{N}|=(1+\escpr{\nabla u +F,b(\nabla u +F)})^{1/2},
\end{equation}
and
\[
|\tilde{N}_h|=\escpr{\nabla u+F, b(\nabla u+F)}^{1/2},
\]
where $\escpr{,}$ is the Euclidean Riemannian metric in $\rr^{2}$, and so
\begin{equation}
\label{eq:modnh}
|N_h|=\frac{|\tilde{N}_h|}{|\tilde{N}|}=\frac{\escpr{\nabla u+F, b(\nabla u+F)}^{1/2}}{\big(1+\escpr{\nabla u +F,b(\nabla u +F)}\big)^{1/2}}.
\end{equation}
Observe that, from \eqref{eq:tilden} and \eqref{eq:modtilden} we also get
\begin{equation}
\label{eq:nxt}
g(N,T)=-\frac{1}{\big(1+\escpr{\nabla u+F,b(\nabla  u+F)}\big)^{1/2}}.
\end{equation}
Hence we obtain from \eqref{eq:area}, \eqref{eq:dgu}, \eqref{eq:jac} and \eqref{eq:modnh}
\begin{equation}
\label{eq:agu}
A(G_u)=\int_\Om \escpr{\nabla u+F, b(\nabla u+F)}^{1/2}\,\frac{\det(g_{ij}+(\nabla u+F)_i(\nabla u+F)_j)^{1/2}}{\big(1+\escpr{\nabla u +F,b(\nabla u +F)}\big)^{1/2}}
\,d\mathcal{L}^{2}.
\end{equation}

We can compute the mean curvature of a graph $G_u$ \cite[Lemma~4.2]{Ga-Ri}

\begin{lemma}
\label{1variation:graph} 
Let us consider the contact sub-Riemannian manifold $(\mathbb{R}^{3},g_{\mathcal{H}_0},\omega_0)$, where $\omega_0$ is the standard contact form in $\mathbb{R}^{3}$ and $g_{\mathcal{H}_0}$ is a positive definite metric in  the horizontal distribution $\mathcal{H}_0$. Let $u:\Omega\subset \rr^{2}\rightarrow \rr$ be a ${C}^2$ function.
We denote by $g=(g_{ij})_{i,j=1, 2}$ the metric matrix and by $b=g^{-1}=(g^{ij})_{i,j=1,2}$ the inverse metric matrix. Then the mean curvature of the graph $G_u$, computed with respect to the downward pointing normal, is given by
\begin{equation}
\label{eq:hgu}
-\divv\bigg(\frac{b(\nabla u+F)}{\escpr{\nabla u+F,b(\nabla u+F)}^{1/2}}\bigg)+\mu,
\end{equation}
where $\mu$ is a bounded function in $\Om\setminus\Om_0$, and $\divv$ is the usual Euclidean divergence in $\Om$.
\end{lemma} 

Furthermore in dimension three we get

\begin{lemma}\label{equivcurvaturamedia}
Let us consider the contact sub-Riemannian manifold $(\mathbb{R}^{3},g_{\mathcal{H}_0},\omega_0)$, where $\omega_0$ is the standard contact form in $\mathbb{R}^{3}$ and $g_{\mathcal{H}_0}$ is a positive definite metric in  the horizontal distribution $\mathcal{H}_0$. Let $u:\Omega\subset \rr^{2}\rightarrow \rr$ be a ${C}^2$ function. Then
\[
\divv\bigg(\frac{b(\nabla u+F)}{\escpr{\nabla u+F,b(\nabla u+F)}^{1/2}}\bigg)=det(g)\divv\bigg(\frac{\nabla u+F}{\escpr{\nabla u+F,\nabla u+F}^{1/2}}\bigg)+\rho,
\]
where $\rho$ is a bounded function in $\Om\setminus\Om_0$.
\end{lemma}

\begin{proof}
The proof is a standard computation. We only note that $\rho$ is of the form 
\[
\frac{\rho_1(b)(u_x-y)^3+\rho_2(b)(u_y+x)^3+\rho_3(b)(u_x-y)^2(u_y+x)+\rho_4(b)(u_x-y)(u_y+x)^2}{(g^{11}(u_x-y)^2+2g^{12}(u_x-y)(u_y+x)+g^{22}(u_y+x)^2)^{3/2}},
\]
where $\rho_i(b)$ are sums and products of the coefficients $g^{ij}$.
\end{proof}

Let $\Sg\subset M$ a $\mathcal{C}^2$ surface and let $p\in\Sg_0$. Then there exists a neighborhood $U$ of $p$ that is a Darboux chart and $\Sg$ can be view as a graph $G_u$ in $(\rr^{3}, g_{\hhh_0},\omega_0)$ above defined. The projection $\Om_0$ of the singular set in $(G_u)_0$ do not depend by the metric $g_{\hhh_0}$. The characteristic curves in $G_u$ with respect to $g_{\hhh_0}$ and the standard Heisenberg metric $g_0$ coincide, as they are determined by $TG_u\cap \hhh$. This implies

\begin{theorem}\label{singularsetth}Let $\s$ be a $\mathcal{C}^2$ oriented immersed surface with constant mean curvature $H$ in $(M,g_\hhh,\omega)$. Then the singular set $\s_{0}$ consists of isolated points and $\mathcal{C}^1$ curves with non-vanishing tangent vector. Moreover, we have 
\begin{enumerate} 
\item[(i)]if $p\in\s_{0}$ is isolated then there is $r > 0$ and $\lambda\in 
\mathbb{R}$ with $|2\lambda| = |H |$ such that the set described as 
\begin{equation*}
D_{r} (p) = \{\gamma^{\lambda}_{p,v}  (s)| v \in T_{p}\s, |v| = 1, s \in [0, r)\}, 
\end{equation*}
is an open neighborhood of $p$ in $\s$, where $\gamma^{\lambda}_{p,v}$ denote the characteristic curve starting from $p$ in the direction $v$ with curvature $\lambda$ \eqref{eq:carcurvesph};
\item[(ii)] if  $p$ is not isolated, it is contained in a $\mathcal{C}^1$ curve 
$\Gamma\subset\s_{0}$. Furthermore there is a neighborhood $B$ of $p$ in $\s$ such that $B-\Gamma$ is the 
union of two disjoint connected open sets $B_{+}$ and $B_{-}$ contained in $\s-\s_{0}$, 
and $\nu_{h}$ extends continuously to $\Gamma$ from both sides of $B-\Gamma$, i.e., the limits 
\begin{equation*}
\nu_{h}^{+}(q)=\lim\limits_{x\rightarrow q, x\in B_{+}} \nu_{h}(x), \hspace{1cm} \nu_{h}^{-}(q)=\lim\limits_{x\rightarrow q, x\in B_{-}} \nu_{h}(x)
\end{equation*}
exist for any $q\in \Gamma\cap B$. These extensions satisfy $\nu_{h}^{+}(q)=-\nu_{h}^{-}(q)$. Moreover, there are exactly two characteristic curves $\gamma^{\lambda}_{1}\subset B_{+}$ and  
$\gamma^{\lambda}_{2}\subset B_{-}$ starting from  $q$ and meeting transversally $\Gamma$ at $q$ with initial velocities $(\gamma^{\lambda}_{1})^{\prime}(0)=-(\gamma^{\lambda}_{2})^{\prime}(0)$. The curvature $\lambda$ does not depend on $q$ and satisfies $| \lambda | = |H|$.
\end{enumerate}
\end{theorem}

\begin{proof} By  \cite[Theorem~B]{CJHMY}, Lemma \ref{1variation:graph} and Lemma \ref{equivcurvaturamedia}, $\s_{0}$ consists of isolated points and $\mathcal{C}^1$ curves with non-vanishing tangent vector. Also (i) follows easily.

Writing 
\[
\nu_h=\frac{b(\n u+F)}{\escpr{\n u+F,b(\n u+F)}^{1/2}}
\]
because of \cite[Theorem~3.10, Corollary~3.6]{CJHMY}, we get (ii).
\end{proof}

\begin{corollary}\label{charmeetsingular} Let $\s$ be a $\mathcal{C}^2$  minimal surface with singular set $\s_{0}$. Then $\s$ is area stationary if and only if the characteristic curves meet the singular curves orthogonally with respect the metric $g$. 
\end{corollary}

The proof is a straightforward adaptation of the Heisenberg one, \cite[Theorem~4.16]{Ri-Ro}. Another version of the last corollary is presented in \cite[Proposition~6.2]{CHY} and \cite[p.~20]{ChengHwang2nd}.

\begin{remark}  \cite[Proposition~4.19]{Ri-Ro} implies that, for $\s$ a $\mathcal{C}^2$ oriented immersed area-stationary surface (with or without a volume constraint), any singular curve of $\s$ is a $\mathcal{C}^2$ smooth curve. 
\end{remark}

\begin{remark}
Another approach to characterize the local behavior of  the singular set is provided in \cite{Na}, where the author constructs a circle bundle over the surface and studies the projection of the singular set.   
\end{remark}

Now we are able to generalize \cite[Theorem~E]{CJHMY} to general three-dimensional contact sub-Riemannian manifolds.

\begin{theorem}\label{genusbound}
Let $\s$ be a $\mathcal{C}^2$ closed, connected surface immersed in a three-dimensional contact sub-Riemannian manifold $M$, with bounded mean curvature. Then $g(\s)\leq 1$, where $g(\s)$ denote the genus of $\s$.
\end{theorem}

\begin{proof} By Theorem \ref{singularsetth} the singular set $\s_0$ consists of  singular curves and isolated singular points. The  line field associated to the characteristic foliation, extended to the singular curves, has a  contribution to the index only due to the isolated singular points, Theorem \ref{singularsetth}.  Now consider a partition of unity $\{\eta_i\}_{i\in I}$ subordinate to a covering of $\s$ with Darboux's charts $\{U_i\}_{i\in I}$. By \cite[Lemma~3.8]{CJHMY} the index associated to the characteristic line field with respect to the Heisenberg metric in the Darboux coordinates is 1 and follows that the index of the vector field 
\[
\sum\limits_{i\in I} \eta_i( \varphi^{-1}_{i}(Z_0) )
\]
is 1 in each singular point of $\s$, since the Darboux's diffeomorphism preserve the index, \cite[Lemma~27, p.~446]{spivak1}. Here $\varphi_i$ denotes the Darboux's diffeomorphism in each chart $U_i$ and $Z_0$ denote the characteristic vector field associate to the Heisenberg metric in $\varphi_i(U_i)$. We get $\chi(\s)\geq 0$, by the Hopf index Theorem, \cite{spivak1}. 

On the other hand for a closed surface $\chi(\s)=2-2g(\s)$, which implies $g(\s)\leq 1$. 
\end{proof}

\begin{remark} When $\s$ is a compact $\mathcal{C}^2$ surface without boundary in a three-dimensional pseudo-hermitian sub-Riemannian manifold, Theorem I in \cite{CJHMY2} implies immediately $\chi(\s)\geq 0$. Then $g(\s)\leq 1$.
\end{remark}

\section{The first variation formula for $\mc{C}^2_\hhh$ surfaces}
\label{sec:1c2h}

Now we present a first variation formula for surfaces of class $\mathcal{C}^2_\hhh$ using variations supported in the non-singular set. Given a surface $\s$ of class $\mathcal{C}^2_\hhh$, we can express $\s$ as the zero level set of a function $f\in\mc{C}^2_\hhh$ with non-vanishing horizontal gradient. 

\begin{remark}\label{convergenzaHhei} In $\hh^1$, by \cite[Proposition~1.20]{Fo-St}, see also \cite[Lemma~2.4]{Ba-SC-Vi} and the proof of Theorem 6.5, step 1, in \cite{FSSC}, the family of smooth surfaces $\{\s_j\}_{j\in\nn}=\{p\in M:f_j(p)=0\}$, where $f_j:=\rho_j \ast f$ and $\rho_j$ are the standard Friedrichs' mollifiers, converge to $\s$ on compact subsets of $\s-\s_0$. Furthermore also the second order derivatives of $f_j$ with respect to horizontal vectors fields converge to the second derivatives of $f$.  
We denote by $Z_j,(\nu_h)_j$ and $N_j$, respectively, the characteristic vector field, the horizontal unit normal and the unit normal of $\s_j$. Furthermore let $(Z_j(p),S_j(p))$ be an orthonormal basis of $T_p\s_j$. We have that
\[
\nu_h=\frac{(Xf)X+(Yf)Y}{\sqrt{(Xf)^2+(Yf)^2}}, \quad Z=\frac{-(Yf)X+(Xf)Y}{\sqrt{(Xf)^2+(Yf)^2}}
\]
and
\[
(\nu_h)_j=\frac{(Xf_j)X+(Yf_j)Y}{\sqrt{(Xf_j)^2+(Yf_j)^2}}, \quad Z_j=\frac{-(Yf_j)X+(Xf_j)Y}{\sqrt{(Xf_j)^2+(Yf_j)^2}}
\]
so $Z_j$ ( resp. $(\nu_j)$) converges to $Z$ (resp. $\nu_h$) with theirs horizontal derivatives. On the other hand
\[
N=\frac{(Xf)X+(Yf)Y+(Tf)T}{\sqrt{(Xf)^2+(Yf)^2+(Tf)^2}}
\]
and
\[
N_j=\frac{(Xf_j)X+(Yf_j)Y+(Tf_j)T}{\sqrt{(Xf_j)^2+(Yf_j)^2+(Tf_j)^2}},
\]
which implies that $N_j$ (resp. $S_j$) converges to $N$ (resp. $S$) but there are not convergence of theirs derivatives.

\end{remark}

\begin{lemma}\label{ConvergenzaH}
Let $\s$ be a $\mc{C}^1$ surface immersed in $M$, such that the derivative in the $Z$-direction of $\nu_h$ exists and is continuous. Assume $\s_0=\emptyset$. Then exists a family of smooth surfaces $\{\s_j\}_{i\in\nn}$ such that 
\[
\lim\limits_{j\rightarrow +\infty}g(\n_{Z_j}(\nu_h)_j,Z_j)=g(\n_Z\nu_h, Z)
\]
uniformly on compact subsets of $\s$.
\end{lemma}

\begin{proof} It is sufficient to prove the result locally in a Darboux's chart. So we consider $\s$ in $(\rr^3,\hhh_0,g_{\hhh_0})$, where $g_{\hhh_0}$ is an arbitrary positive definite smooth metric in $\mathcal{H}_0$. We denote by $(\nu_h)_{0}$ the horizontal unit normal with respect the Heisenberg metric $g_0$. By Remark \ref{convergenzaHhei} the statement holds in the Heisenberg metric. 
As in \eqref{eq:tilden} and \eqref{eq:modtilden} we have
\[
(\nu_h)_j=\frac{(g^{11}X(f_j)+g^{12}Y(f_j))X+(g^{12}X(f_j)+g^{22}Y(f_j))Y}{\sqrt{\escpr{(X(f_j),Y(f_j)),b(X(f_j),Y(f_j))}}}
\]
and
\[
((\nu_h)_j)_{0}=\frac{X(f_j)\,X+Y(f_j)\,Y}{\sqrt{\escpr{(X(f_j),Y(f_j)),b(X(f_j),Y(f_j))}}}.
\]
Similar expressions hold for $\nu_h$ and $(\nu_h)_0$. The $Z$-direction does not depend on the metric, since it is determined by $T\Sg\cap\hhh$. Furthermore, since the coefficients $g^{il}$ are smooth, the convergence also holds in the arbitrary metric.

\end{proof}

Now we are able to prove

\begin{proposition}\label{1var:c2h} Let $\s$ be an oriented immersed $\mathcal{C}^2_\hhh$ surface in a contact sub-Riemannian three-dimensional manifold $(M,g_\hhh,\omega)$. Then the first variation of the area induced by the vector field  $U=f \nu_h+lZ+hT$, with  $f,l,h\in\mathcal{C}^1_{0}(\s-\s_0)$, is 
\begin{equation*}
\frac{d}{ds}\Big{|}_{s=0} A(\varphi_{s}(\s))=-\m_{\s} g(U,N)\,Hd\s.
\end{equation*}
\end{proposition}

\begin{proof} Due to the linearity of  \eqref{var} respect  to $U$ we can compute separatly the variations in the direction of $Z$, $\nu_h$ and $T$. By \eqref{var} the variation in the direction of $\nu_h$ becomes
\begin{equation*}
\begin{aligned}
\m_{\s}f|N_{h}|g(\n_{Z}\nu_h,Z) d\s.
\end{aligned}
\end{equation*}  
The variation produced by $T$ is \eqref{var}
\[
\m_{\s}\{-S(h)+hg(N,T)|N_h|g(\tau(Z),Z)\}d\s=\m_{\s}hg(N,T)g(\n_{Z}\nu_h,Z) d\s,
\]
because of 
\begin{equation*}
\begin{aligned}
\m_{\s}\{-S(h)&+hg(N,T)|N_h|g(\tau(Z),Z)\}d\s=\lim\limits_{j\rightarrow +\infty}\m_{\s_j}\{-S_j(h)+hg(N_j,T)|(N_h)_j|g(\tau(Z_j),Z_j)\}d\s_j\\
&=\lim\limits_{j\rightarrow +\infty} \Big\{ \m_{\s_j}hg(N_j,T) \,g(\n_{Z_j}\nu^j_h,Z_j)d\s_j-\m_{\s_j} \divv_{\s_j} (hg(N_j,T)^2S_j)d\s_j\Big\} \\
&=\m_{\s}hg(N,T) \,g(\n_{Z}\nu_h,Z)d\s,
\end{aligned}
\end{equation*}
where we have used the Riemannian divergence theorem in the last equality. In an analog way
\begin{equation*}
\begin{aligned}
\m_{\s}\{&c_1g(N,T)g(J(\nu_h),Z)+|N_h|Z(l)\}d\s=\lim\limits_{j\rightarrow +\infty}\m_{\s_j}\{c_1g(N_j,T)g(J((\nu_h))_j,Z_j)+|(N_h)_j|Z_j(l)\}d\s_j.
\end{aligned}
\end{equation*}
Now since $|N_h|Z(l)=Z(|N_h|l)-lZ(|N_h|)$, by Lemma \ref{conti} and Lemma \ref{devergenceteorem} we get
\begin{equation*}
\begin{aligned}
\m_{\s}\{&c_1g(N,T)g(J(\nu_h),Z)+|N_h|Z(l)\}d\s=\lim\limits_{j\rightarrow +\infty}\m_{\s_j}\divv_{\s_j}(l|(N_h)_j|Z_j)d\s_j=0,
\end{aligned}
\end{equation*}
so we have proved that the variation produced by $Z$ vanishes. Since $g(U,N)=f|N_h|+hg(N,T)$ we finally get 
\begin{equation*}
\frac{d}{ds}\Big{|}_{s=0} A(\varphi_{s}(\s))=-\m_{\s} g(U,N)\,Hd\s.
\end{equation*}
\end{proof}

\begin{remark}\label{osservazione:regolaritaottimale} Proposition \ref{1var:c2h} holds also for a $\mc{C}^1$ surface $\s$ in which $\nu_h$ (or equivantelly $Z$) is $\mc{C}^1$ in the $Z$-direction. It does not imply the $\mc{C}^2_\hhh$ regularity when $T\s\cap\hhh$ has dimension one. We thank F. Serra Cassano for pointing out this fact.  
\end{remark}

\section{Second variation formulas }

In this section we will compute a second variation formula for a minimal surface considering variations in the direction of $N$ and $T$ in the regular part and variation induced by the Reeb vector field supported near the singular set of the surface. We restrict ourself to the case of pseudo-hermitian manifolds. Consider the orthonormal basis $\{Z,\nu_{h},T\}$, we can compute
\begin{equation}\label{liebrachet}
\begin{aligned}
&[Z,\nu_{h}]= c_1T+\theta(Z)Z+\theta(\nu_{h})\nu_{h} \\
&[Z,T]=g(\tau(Z),Z)Z+\{g(\tau(Z),\nu_h)+\theta(T)\}\nu_{h}\\
&[\nu_{h},T]=  \{ g(\tau(Z),\nu_h) -\theta(T)\}Z+g(\tau(\nu_h),\nu_h)\nu_{h} 
\end{aligned}
\end{equation}
where $\theta$ defined in \ref{def:theta} and  we have computed $\theta(T)=-1+g(D_{T}\nu_{h},Z)$ using \eqref{phvsLC}. 

\begin{lemma} Let $\s$ be a $\mc{C}^2$ immersed oriented surface with constant mean curvature $H$  in a \phMM. We consider a point $p\in\s-\s_0$ and we denote by $\alpha:I\rightarrow \s-\s_0$ the integral curve of $S_p$. Then the results in Proposition \ref{jacobifields} hold with $U_{\alpha(\eps)}=Z_{\alpha(\eps)}$. Furthermore in 
$\s-\s_{0}$, the normal vector $N$ is $\mathcal{C}^{\infty}$ in the direction of the characteristic field $Z$.
\end{lemma}

\begin{proof} From (i) in Proposition \ref{jacobifields} and from \eqref{eq:carcurvesph} follows that $V_\eps$ and $\dot{\ga_\eps}$ are $\mc{C}^{\infty}$ along characteristic curves and we express the unit normal to $\s$ along $\ga_\eps$ by
\[
N=\pm\frac{\dot{\ga_\eps}\times V_\eps}{|\dot{\ga_\eps}\times V_\eps|},
\]
where $\times$ denote the cross product in $(M, g)$. We conclude that $N$ is $\mc{C}^{\infty}$ along $\ga_\eps$.
\end{proof}

\subsection{Second variation in the regular set}
Now we present a variation formula in the regular part of the surface induced by a vector field of the form $vN+wT$. 

\begin{lemma} Let $\Sigma$ be a $\mathcal{C}^2$ surface of constant mean curvature $H$ in a \phMM. Then we have
\begin{equation}\label{curvaturaT}
g(R(T,Z)\nu_h,Z)=-\nu_h(g(\tau(Z),Z))+Z(g(\tau(Z),\nu_h))-2\omega(\nu_h)g(\tau(Z),\nu_h)+2Hg(\tau(Z),Z).
\end{equation}
\end{lemma}

\begin{proof}  By \eqref{phvsLC} it is not difficult show (see \cite[Theorem~1.6]{Dr-To} for the case in which $c_1=2$)
\begin{equation}\label{equivRTRLCT}
g(R(T,Z)\nu_h,Z)=g(R^{LC}(\nu_h,Z)T,Z),
\end{equation}
where $R^{LC}$ is the curvature tensor with respect to the Levi-Civita connection, that can be easily computed as 
\[
g(R^{LC}(\nu_h,Z)T,Z)=-\nu_h(g(\tau(Z),Z))+Z(g(\tau(Z),\nu_h))-2\omega(\nu_h)g(\tau(Z),\nu_h)+2Hg(\tau(Z),Z),
\]
take in account \eqref{phvsLC} and \eqref{def:sigma}.
\end{proof}

\begin{theorem}\label{seconvariation}
Let $\Sigma$ be a $\mathcal{C}^2$ minimal surface in a \phMM, with singular set $\Sigma_{0}$. We consider a $\mathcal{C}^1$ vector field $U=vN+wT$, where $N$ is the unit normal vector to $\Sigma$
 and $w,v\in \mathcal{C}^1_{0}(\Sigma-\s_{0})$. Then the second derivative of the area for the variation induced by  $U$ is given by
 \begin{equation}
\label{ secondaregular}
A''(0)=\m\limits_{\Sigma} \{|N_{h}|^{-1} Z(u)^2 +u^2 q \}  d\Sigma+\m\limits_{\Sigma} div_{\Sigma}(\xi Z+\zeta Z+\eta S)d\Sigma,
\end{equation}
with
\begin{equation*}
\xi=g(N,T)\{|N_{h}|\theta(S)+c_1g(N,T)^2+(1+g(N,T)^2)g(\tau(Z),\nu_h)\}u^2, 
\end{equation*}
\begin{equation*}
\zeta=|N_h|^2|\{g(N,T)(|N_{h}|\theta(S)+c_1g(N,T)^2+(1+g(N,T)^2)g(\tau(Z),\nu_h))w^2  -2g(B(Z),S)vw  \}, 
\end{equation*}
\begin{equation*}
\eta=(|N_h|^2v^2-(g(N,T)v+w)^2)g(\tau(Z),Z)
\end{equation*}
and 
\begin{equation*}
\begin{aligned}
q&=|N_{h}|\{ -W+c_1^2+c_1g(\tau(Z),\nu_h)   \}-|N_{h}| (|N_{h}|(c_1+g(\tau(Z),\nu_h))-\theta(S))^2\\
&+g(N,T)g(R(Z,T)\nu_{h},Z)-g(N,T)Z(g(\tau(Z),\nu_h)),
\end{aligned}
\end{equation*}
where $u=g(U,N)$, $R$ is the pseudo-hermitian curvature tensor and $B$ is the Riemannian shape operator.

\end{theorem}

\begin{remark} If $\s$ is area stationary without boundary, then 
\begin{equation*}
\m_{\s}div_{\s}(fS)=0,
\end{equation*}
for every $f\in\mathcal{C}^1_{0}(\s)$, by  Corollary \ref{charmeetsingular}.
\end{remark}

\begin{proof} (of Theorem \ref{seconvariation}) We can reason as in the proof of the first variation formula. We have 
\begin{equation*}
U(|V|)=\frac{1}{|V|} g(\n_U V,V)
\end{equation*}
and\begin{equation*}
U(U(|V|))=-\frac{1}{|V|^3}g(\n_{U}V,V)^2 +\frac{1}{|V|}\left(g(\n_{U}\n_{U}V,V)+|\n_{U}V|^2\right).
\end{equation*}
We fix 
\begin{equation*}
\lambda:= g(\n_{U}V,V),
\end{equation*}
so we get
\begin{equation}\label{2derivataastratta}
U(U(|V|))=\frac{1}{|V|}\left\{\underbrace{ g(\n_{U}\n_{U}V,V)}_{II}+\underbrace{g(\n_{U}V,\n_{U}V -\frac{\lambda}{|V|^2}V)}_{I}\right\}.
\end{equation}
As $V=g(E_{1},T)E_{2}-g(E_{2},T)E_{1}$ we compute
\begin{equation}
\begin{aligned}
\n_{U}V(0)&=U(g(E_{1},T))S-U(g(T,E_{2}))Z+|N_{h}|\n_{U}E_{1}.
\end{aligned}
\end{equation}
Observing $g(E_1,T)=0$ and
\begin{equation*}
-\frac{\lambda V(0)}{|V(0)|^2}=-g(\n_{U}V,Z)Z,
\end{equation*}
we have
\begin{equation}
\n_{U}V-\frac{\lambda V(0)}{|V(0)|^2}=U(g(E_{1},T))S+|N_{h}|g(\n_{U}E_{1},S)S+|N_{h}|g(\n_{U}E_{1},N)N.
\end{equation}
By $|N_{h}|S=g(N,T)N-T$ and $|N_{h}|N+g(N,T)S=\nu_{h}$ we obtain
\begin{equation*}
\begin{aligned}
I&=g(\n_{U}E_{1},N)^2.
\end{aligned}
\end{equation*}
Using (i) in lemma \ref{conti}, $T=g(N,T)N-|N_h|S$ and $\nu_h=g(N,T)S+|N_h|N$ we have
\begin{equation}
\begin{aligned}
I&=\{Z(g(U,N)+g(N,T)|N_h|g(U,N)(c_1+g(\tau(Z),\nu_h))-g(U,S)|N_h|\theta(S)\}^2
\end{aligned}
\end{equation}
and
\begin{equation}\tag{I}
\begin{aligned}
|N_h|^{-1}I&=|N_h|^{-1}Z(u)^2+2g(N,T)uZ(u)(c_1+g(\tau(Z),\nu_h))\\
&+g(N,T)^2|N_h|u^2(c_1+g(\tau(Z),\nu_h))^2+2Z(u)|N_h|w\theta(S)\\
&+2g(N,T)|N_h|^2uw\theta(S)(c_1+g(\tau(Z),\nu_h))+w^2|N_h|^3\theta(S)^2,
\end{aligned}
\end{equation}
where $u=g(U,N)$.

Now we consider
\begin{equation*}
\begin{aligned}
|N_h|^{-1}II&=\underbrace{g(R(V,U)U,Z)}_{A}+\underbrace{g(\n_{V}\n_{U}U,Z)}_{B}+\underbrace{g(\n_{[U,V]}U,Z)+g(\n_{U}[U,V],Z)}_{C}\\
&+g(\underbrace{\n_{U}Tor_{\n}(U,V),Z)}_{D},
\end{aligned}
\end{equation*} 
as
\begin{equation*}
\begin{aligned}
g(\n_{U}\n_{V}U,V)&=g(R(V,U)U,V)+g(\n_{V}\n_{U}U,V)+g(\n_{[U,V]}U,V).
\end{aligned}
\end{equation*} 
By equation (\ref{phvsLC}) we obtain
\begin{equation}\label{derUU}
\begin{aligned}
\n_U U=-g(U,\nu_h)^2g(\tau(Z),Z) T-g(U,T)g(U,\nu_h)(c_1Z+\tau(\nu_h)) ,
\end{aligned}
\end{equation} 
furthermore by  (i) in Lemma \ref{conti} and $D_{U}U=0$ we have
\begin{equation*}
\begin{aligned}
B=&- Z(v(g(N,T)v+w)|N_{h}|^2(c_1g(\tau(Z),\nu_h)))\\
&- g(N,T)Z(g(N,T))v(g(N,T)v+w)(c_1g(\tau(Z),\nu_h))
\end{aligned}
\end{equation*} 
and for the linearity of $R$ we get
\begin{equation*}
\begin{aligned}
A&=|N_{h}|^2v (R(Z,U)\nu_{h},Z).
\end{aligned}
\end{equation*} 
On the other hand as $[U,V]=U(g(E_1,T))E_2-U(g(E_2,T))E_1$ and $\s$ stationary we have
\begin{equation*}
\begin{aligned}
C&= (Z(g(U,T))+c_1g(U,\nu_h))(2g(U,\nu_h) \theta(S)-(g(N,T)g(U,T)-|N_h|g(U,\nu_h))g(\tau(Z),\nu_h)) \\
&   -S(g(U,T))g(U,T)g(\tau(Z),Z)     -U(U(g(T,E_{2})))                  
\end{aligned}
\end{equation*} 
and writing $-U(U(g(T,E_{2}))$ as
\[
-S(g(\n_U U,T))+c_1g(N,T)g(\n_U U,Z)-c_1g(\n_{S}U,J(U))-c_1g(Tor_{\n}(U,S),J(U))
\]
that is
\begin{equation*}
\begin{aligned}
&-S(g(U,\nu_h)^2g(\tau(Z),Z)) -c_1g(N,T)g(U,T)g(U,\nu_h)(2g(\tau(Z),\nu_h))\\
&+c_1g(U,\nu_h)(g(N,T)g(U,T)-|N_h|g(U,\nu_h))g(\tau(Z),\nu_h)-c_1g(U,\nu_h)^2\theta(S),
\end{aligned}
\end{equation*} 
we get that $C$ equals
\begin{equation*}
\begin{aligned}
&Z(g(N,T)v+w)ug(\tau(Z),\nu_h)+c_1|N_h|^2v^2\theta(S)-c_1g(N,T)|N_h|(g(N,T)v+w)v(2g(\tau(Z),\nu_h))\\
&+2 Z(g(N,T)v+w)|N_h|v \theta(S)+ S(g(U,\nu_h)^2g(\tau(Z),Z))    -S(g(U,T))g(U,T)g(\tau(Z),Z) .
\end{aligned}
\end{equation*}

Now $V\in\mathcal{H}$ implies $D=-|N_h|g(U,\nu_h)^2(g(\tau(Z),Z))^2+g(U,T)g(\n_U \tau(V),Z )$. On the other hand
\begin{equation*}
\begin{aligned}
g(\n_{U}\tau(V),Z)=g(\n_{U}\tau(V)-\n_{V}U,Z)+g(\n_{V}\tau(U),Z),
\end{aligned}
\end{equation*} 
where $g(\n_{V}\tau(U),Z)=Z(|N_{h}|g(U,\nu_h)g(\tau(Z),\nu_h))-Z(|N_h|)g(U,\nu_h)g(\tau(Z),\nu_h)$ and
\begin{equation*}
\begin{aligned}
g(\n_{U}\tau(V)-\n_{V}\tau(U),Z)=g((\n_{U}\tau)V-(\n_{V}\tau)U,Z)+g(\tau(Z),[U,V]+Tor(U,V)).
\end{aligned}
\end{equation*} 
We have that $g((\n_{U}\tau)V-(\n_{V}\tau)U,Z)$ is equal to
\begin{equation*}
\begin{aligned}
&|N_{h}|g(U,\nu_h)g((\n_{\nu_{h}}\tau)Z-(\n_{Z}\tau)\nu_{h},Z)+g(U,T)|N_{h}|g((\n_{T}\tau)Z,Z)
\end{aligned}
\end{equation*} 
and by Theorem 1.6 in \cite{Dr-To} and the fact that $(\n_{X}\tau)Y$ is a tensor we obtain
\begin{equation*}
\begin{aligned}
g((\n_{U}\tau)V-(\n_{V}\tau)U,Z)&=-|N_{h}|g(U,\nu_h)g(R(T,Z)\nu_{h},Z)+|N_{h}|g(U,T)(T(g(\tau(Z),Z))\\
&+2g(\tau(Z),\nu_h) \omega(T)).
\end{aligned}
\end{equation*} 
Finally since $g(Tor(U,V),\tau(Z))=g(U,T)|N_{h}|(g(\tau(Z),\nu_h)^2+(g(\tau(Z),Z))^2)$ and writing $g(\tau(Z),[U,V])$ as
\begin{equation*}
\begin{aligned}
&-S(g(U,T))g(U,T)g(\tau(Z),Z)+g(N,T)g(\tau(Z),\nu_h) (Z(g(U,T))+c_1g(U,\nu_h))
\end{aligned}
\end{equation*} 
together with (i) in Lemma \ref{conti} we obtain that $D$ equals 
\begin{equation*}
\begin{aligned}
&-|N_h|g(U,\nu_h)^2(g(\tau(Z),Z))^2+|N_h|(g(N,T)v+w)^2(T(g(\tau(Z),Z))+g(\tau(Z),\nu_h)^2\\
&+(g(\tau(Z),Z))^2)-|N_h|^2(g(N,T)v+w)vg(R(T,Z)\nu_h,Z) +2\omega(T)g(\tau(Z),\nu_h) \\
&+|N_h|(g(N,T)v+w)Z(|N_h|vg(\tau(Z),\nu_h)) +c_1g(N,T)|N_h|v(g(N,T)v+w)g(\tau(Z),\nu_h) \\
&-S(g(U,T))g(U,T)g(\tau(Z),Z)+g(N,T)Z(g(N,T)v+w)(g(N,T)v+w)g(\tau(Z),\nu_h).
\end{aligned}
\end{equation*} 

The sum of all terms that contain $g(\tau(Z),Z)$, after have used lemma \ref{devergenceteorem} is 
\begin{equation*}
div_{\s}((|N_h|^2v^2-(g(N,T)v+w)^2)g(\tau(Z),Z) S)+g(N,T)(g(N,T)v+w)^2\nu_h(g(\tau(Z),Z)).
\end{equation*}

By the definition of $\theta$ we have
\begin{equation}\label{deromega}
\begin{aligned}
Z(\theta(S))&=g(R(S,Z)\nu_{h},Z)+\theta([Z,S])\\
&=g(R(S,Z)\nu_{h},Z)+g(N,T)\theta(S)^2-g(N,T)|N_{h}|\theta(S)(c_1+g(\tau(Z),\nu_h)),
\end{aligned}
\end{equation}
where we used that $[Z,S]$ is tangent to $\s$. We note that since $\n_{Z} \nu_h\equiv 0$ \eqref{deromega} make sense when $\s$ is of class $\mc{C}^2$. Furthermore by Lemma \ref{devergenceteorem} and equation (\ref{deromega}) we have that $2Z(u)|N_h|w\theta(S)+2Z(g(N,T)v+w)|N_h|v\theta(S)$ is equal to
\begin{equation*}
\begin{aligned}
&\textcolor{black}{div_{\s}(g(N,T)|N_h|(v^2+w^2)\theta(S)Z)}+\frac{Z(g(N,T))}{|N_h|}\theta(S)(v^2+w^2)\\
&\textcolor{black}{+2div_{\s}(|N_h|vw\theta(S)Z)}-2vwZ(|N_h|)\theta(S)-2vw|N_h|g(R(S,Z)\nu_h,Z)\\
&-g(N,T)|N_h|(v^2+w^2)g(R(S,Z)\nu_h,Z)
\end{aligned}
\end{equation*}
and similary $B$ equals
\begin{equation*}
\begin{aligned}
&c_1g(N,T)|N_h|v(g(N,T)v+w)(c_1g(\tau(Z),\nu_h))\textcolor{black}{-div_{\s}(|N_h|^2v(g(N,T)v+w)(c_1g(\tau(Z),\nu_h))Z)}.
\end{aligned}
\end{equation*}
In the same way $g(N,T)Z(u^2)(g(\tau(Z),\nu_h))+g(N,T)^2|N_h|u^2(c_1+g(\tau(Z),\nu_h))^2$ can be expressed 
\begin{equation*}
\begin{aligned}
&\textcolor{black}{div_{\s}(g(N,T)u^2(g(\tau(Z),\nu_h))Z)-g(N,T)u^2Z(g(\tau(Z),\nu_h))}\textcolor{black}{+u^2|N_h|c_1(c_1+g(\tau(Z),\nu_h))}\\
&+u^2(g(\tau(Z),\nu_h))\theta(S)\textcolor{black}{-u^2|N_h|^3(g(\tau(Z),\nu_h))^2}.
\end{aligned}
\end{equation*}
Furthermore $-Z(g(N,T)v+w)(u+g(N,T)^2v-g(N,T)w) +|N_h|(g(N,T)v+w)Z(|N_h|vg(\tau(Z),\nu_h))$ become
\begin{equation*}
\begin{aligned}
&\textcolor{black}{div_{\s}(g(\tau(Z),\nu_h)(g(N,T)v+w)(v+g(N,T)w)Z)}\\
&+g(N,T)(-\theta(S)-|N_h|(2g(\tau(Z),\nu_h)))g(\tau(Z),\nu_h)(g(N,T)v+w)(v+g(N,T)w)\\
&-g(N,T)(g(N,T)v+w)^2Z(g(\tau(Z),\nu_h))-Z(g(N,T))g(\tau(Z),\nu_h) w(g(N,T)v+w)
\end{aligned}
\end{equation*}
and all the other terms not considered are
\begin{equation*}
\begin{aligned}
\textcolor{black}{|N_h|^{-1}Z(u)^2}+|N_h|^3v^2g(R(Z,\nu_h)\nu_h,Z)-2|N_h|^2v(g(N,T)v+w)g(R(T,Z)\nu_h,Z)
\end{aligned}
\end{equation*}
and
\begin{equation*}
\begin{aligned}
&2g(N,T)|N_h|^2uw\theta(S)(2g(\tau(Z),\nu_h))+|N_h|^3\theta(S)^2w^2+2|N_h|^2v^2\theta(S) \\
&-4g(N,T)|N_h|v(g(N,T)v+w)-|N_h|(g(N,T)v+w)^2(g(\tau(Z),\nu_h)^2-2\omega(T)g(\tau(Z),\nu_h)).
\end{aligned}
\end{equation*}
Since $g(R(Z,\nu_h)\nu_h,Z)=-W$ we have that the terms in which appear the curvature tensor are equal to
\begin{equation}\label{O}
\begin{aligned}
-|N_h|Wu^2+g(N,T)|N_h|^2(w^2-v^2)g(R(T,Z)\nu_h,Z)
\end{aligned}
\end{equation}
and by equation (\ref{curvaturaT}) we have \eqref{O} sums with $g(N,T)(g(N,T)v+w)^2(\nu_h(g(\tau(Z),Z))+2\omega(\nu_h)g(\tau(Z),\nu_h)-Z(g(\tau(Z),\nu_h)))$ is
\begin{equation*}
\begin{aligned}
-|N_h|Wu^2-g(N,T)u^2g(R(T,Z)\nu_h,Z).
\end{aligned}
\end{equation*}
Finally a long but standard computation shows that the remaining terms add up to
 \begin{equation*}
\begin{aligned}
&c_1u^2|N_{h}|(c_1+g(\tau(Z),\nu_h))-|N_{h}| (|N_{h}|(c_1+g(\tau(Z),\nu_h))-\theta(S))^2u^2-g(N,T)u^2Z(g(\tau(Z),\nu_h))\\
&+g(N,T)(g(N,T)v+w)^2(\nu_h(g(\tau(Z),Z))+2\omega(\nu_h)g(\tau(Z),\nu_h)-Z(g(\tau(Z),\nu_h))).
\end{aligned}
\end{equation*}
Since that $u^2=v^2+2g(N,T)vw+g(N,T)^2w^2$, we get the statement.

\end{proof}

\begin{remark} If we suppose that our variation is not by Riemannian geodesics, i.e. we remove the hypothesis $D_{U}U=0$, we get the additional term
\begin{equation}\label{var:nongeodesic}
-\m_{\s} div_{\s}(|N_{h}|g(D_U U,Z)Z)d\s+\m_{\s}div_{\s}(g(D_{U}U,T)S)d\s.
\end{equation}
It is worth mention to remark that \eqref{var:nongeodesic} vanishes when the variation functions $w,v$ have support in the regular set.
 \end{remark}

 \begin{proof} From the term $-U(U(g(E_{2},T)))$ we have $S(g(D_{U}U,T))-c_1g(N,T)g(D_U U,Z)$. Furthermore 
 $$|N_{h}|g(\n_{Z}D_{U}U,Z)=-Z(|N_h|g(D_U U,Z))+g(D_U U,Z)Z(|N_h|)$$
 and $|N_{h}|^{-1}g(D_{U}U,T)g(\tau(V),V)=g(D_{U}U,T)|N_{h}|g(\tau(Z),Z)$. By lemma \ref{conti} we have
 \begin{equation*}
Z(|N_h|)=g(N,T)(|N_h|\theta(S)+c_1-|N_h|^2(2g(\tau(Z),\nu_h)))
\end{equation*}
and, as no others terms are involved, we conclude applying Lemma \ref{devergenceteorem}.
 \end{proof}

\subsection{Second variation moving the singular set}

By Theorem \ref{singularsetth} the singular set of a $\mc{C}^2$ surface is composed of singular curves and isolated singular points without accumulation points. 
First we present a second variation formula induced by a vertical variation near a singular curve, i.e. in a tubular neighborhood of radius $\varepsilon > 0$ of the singular curve that is the union of all the characteristic curves centered at $(\s_0)_c$ defined in the interval $[-\eps,\eps]$.

\begin{lemma}\label{secondvar:scurve} Let $\s$ be a complete  $\mathcal{C}^2$ area-stationary surface immersed in $M$, with a singular curve $\Gamma$ of class $\mathcal{C}^{3}$. Let $w\in\mathcal{C}^{2}_{0}(\s)$. We consider the variation of $\s$ given by $p \rightarrow \exp_{p}(rw(p)T_{p})$. Let $U$ be a tubular neighborhood of $supp(w)\cap\Gamma$, and assume that $w$ is constant along the characteristic curves in $U$. Then there is a tubular neighborhood $U' \subset U$ of $supp(w) \cap \Gamma$ so that
 \begin{equation*}
\begin{aligned}
A''(\varphi_{r}(U'))=\frac{d^2}{dr^2}A(\varphi_r(U'))&=\m_{\s}\{2 w^{2}|N_{h}|(g(\tau(Z),\nu_h)^{2}+g(\tau(Z),Z)^{2}   ) \}d\s\\
&+\m_{\s}div_{\s}(w^{2}g(\tau(Z),Z) S)d\s+\m_{\Gamma}S(w)^{2}d\Gamma.
\end{aligned}
\end{equation*}

\end{lemma}

\begin{proof} We consider the singular curve $\Gamma$ parametrized by arc-length with variable $\varepsilon$. By Theorem \ref{charmeetsingular} we can parametrize $\s$ in a neighborhood of $supp(w)\cap\Gamma$ by $(s,\varepsilon)$, so that the curves with $\varepsilon$ constant are the characteristic curves of $\s$.
 As $E_{i}$ are Jacobi-like vector fields it is easy to prove that $g(E_{i},T)''=0$, so that $g(E_{i},T)=g(E_{i},T)'(0)r+g(E_{i},T)(0)$ and, in particular, we have $g(E_{1},T)=0$. This means that $|V(r)|=|g(E_{2},T)||E_{1}|=|F(p,s,r)||E_{1}|$ which vanishes if and only if $F(p,s,r)=0$. As 
\begin{equation*}
\frac{\partial F(p,0,0)}{\partial s}=-Z(|N_{h}|)=\frac{g(N,T)}{N_{h}}Z(g(N,T))=-c_1,
\end{equation*}
we can apply the implicit function theorem i.e., there exists $s(\varepsilon,r)$ such that the curve $F(p,s(\varepsilon,r),r)=0$ is a graph on $\Gamma(\varepsilon)$. We have obtained
\begin{equation*}
\begin{aligned}
A(\varphi_{r}(U'))&=\m_{-s_{0}}^{s_{0}}\m_{-\varepsilon_{0}}^{\varepsilon_{0}} |F(p,s,r)||E_{1}|dsd\varepsilon\\
&=   \m_{-\varepsilon_{0}}^{\varepsilon_{0}}       \left\{  \m_{-s_{0}}^{s(\varepsilon,r)} F(p,s,r)|E_{1}|  ds  -\m_{s(\varepsilon,r)}^{s_{0}}  F(p,s,r)|E_{1}|       ds\right\}d\varepsilon=  \m_{-\varepsilon_{0}}^{\varepsilon_{0}}     f_{\varepsilon}(r)d\varepsilon
\end{aligned}
\end{equation*}
and
\begin{equation*}
\begin{aligned}
f_{\varepsilon}''(r)&=2\frac{\partial F(p,s(\varepsilon,r),r)}{\partial r}|E_{1}|\frac{\partial s(\varepsilon,r)}{\partial r}+ \m_{-s_{0}}^{s(\varepsilon,r)} (2\frac{\partial F(p,s,r)}{\partial r} \frac{\partial |E_{1}|}{\partial r}+F(p.s,r)\frac{\partial^{2} |E_{1}|}{\partial r^{2}} )ds\\
& -\m_{s(\varepsilon,r)}^{s_{0}}  \Big(2\frac{\partial F(p,s,r)}{\partial r} \frac{\partial |E_{1}|}{\partial r}+F(p,s,r)\frac{\partial^{2} |E_{1}|}{\partial r^{2}} \Big)     ds\\
&=\dot{w}(\varepsilon)^{2}+\m_{-s_{0}}^{s_{0}}2S(w)U(|E_{1}|)+|N_{h}|U(U(|E_{1}|))ds
\end{aligned}
\end{equation*}
 we have that the second variation formula becomes
 \begin{equation*}
\begin{aligned}
A''(\varphi_{r}(U'))&=\m_{(\s_{0})_{c}}S(w)^{2}+\m_{\s}\{2S(w)U(|E_{1}|)+|N_{h}|U(U(|E_{1}|))\}d\s\\
&=\m_{\s}\{ 2wS(w)g(\tau(Z),Z)+w^{2}|N_{h}|(2g(\tau(Z),\nu_h)^{2}+g(\tau(Z),Z)^{2}   )   \}d\s\\
&+\m_{(\s_{0})_{c}}S(w)^{2}\\
&=\m_{\s}\{ 2w^{2}|N_{h}|(g(\tau(Z),\nu_h)^{2}+g(\tau(Z),Z)^{2}   )  \}d\s+\m_{\s}div_{\s}(w^{2}g(\tau(Z),Z) S)d\s\\
&+\m_{(\s_{0})_{c}}S(w)^{2},
\end{aligned}
\end{equation*}
where we have used lemma \ref{devergenceteorem} and 
 \begin{equation*}
\begin{aligned}
U(U(|E_{1}|))&=g(\n_{U}\n_{U}E_{1},E_{1})+g(\n_{U}E_{1},\n_{U}E_{1}-(g(\n_{U}E_{1},E_{1})/|E_{1}|^{2})E_{1})\\
&=wg(\n_{U}\tau(E_{1}),E_{1})+g(\n_{U}E_{1},N)^{2}+g(\n_{U}E_{1},E_{2})^{2}\\
&=Z(w)^{2}+w^{2}\{ 2g(\tau(Z),\nu_h)^{2}+g(\tau(Z),Z)^{2}  \}
\end{aligned}
\end{equation*}
as 
\[
g(\n_U \tau(E_1),E_1)=w(g(\tau(Z),Z)^2+g(\tau(Z),\nu_h)^2)
\]
because of \cite[equation~1.77]{Dr-To}.

\end{proof}

\begin{remark} The hypothesis $\Ga\in\mc{C}^3$ is purely technical. We only need it when we apply the implicit function theorem and it can be weakened. On the other hand in all examples in our knowledge singular curves in area-stationary surfaces are $\mc{C}^{\infty}$.
\end{remark}

Finally we consider a variation which are constant in a neighborhood of the singular set. This hypothesis is reasonable when we move the surface close to isolated singular points, otherwise the second variation blow-up. By a tubular neighborhood of a singular point $q$ we mean the union of all the characteristic segments of length $\varepsilon$ going in $q$ or coming out from $q$.

\begin{lemma}\label{secondvar:spoint} Let $\s$ be a complete $\mathcal{C}^2$ area-stationary surface immersed in $M$ with an isolated singular point $p_0$. Let $w\in\mathcal{C}^{2}_{0}(\s)$. We consider the variation of $\s$ given by $p \rightarrow \exp_{p}(rw(p)T_{p})$. Let $U$ be a tubular neighborhood of $p_0$ and assume that $w$ is constant near $p_0$. Then there is a tubular neighborhood $U' \subset U$ of $p_0$ so that
 \begin{equation*}
\begin{aligned}
A''(\varphi_{r}(U'))&=\m_{\s}\{2 w^{2}|N_{h}|(g(\tau(Z),\nu_h)^{2}+g(\tau(Z),Z)^{2}   ) \}d\s.
\end{aligned}
\end{equation*}
\end{lemma}

The proof is an analog of the previous one using variations moving singular curves. We note that in this case $F(p,s,r)=0$ if and only if $p$ is equal to the original singular point $p_0$ and the statement follows.

\begin{remark} We note that Theorem \ref{seconvariation} coincide with \cite[Theorem~3.7]{Ri-Ro-Hu} in the spacial case of the Heisenberg group $\mathbb{H}^1$ and with \cite[Theorem~5.2]{Ro} for three Sasakian sub-Riemannian manifolds. It can be easily see by (iv) in Lemma \ref{conti}. Furthermore we will see in the next section that (\ref{ secondaregular}) generalize the second variation formula in \cite{CJHMY}.
\end{remark}

\section{Two stability operators.}

The first stability operator which we present gives a criteria for instability in the regular set of a surface. It is the counterpart of the Riemannian one. 

\begin{proposition}\label{stabilityoperatorI} Let $\s$ be a $\mathcal{C}^2$ immersed surface with unit normal vector $N$ and singular set $\s_{0}$ in a pseudo-hermitian manifold $(M,g_\hhh,\omega,J)$. Consider two functions $u\in \mathcal{C}_{0} (\s -\s_{0})$ and $v\in \mathcal{C}_{0} (\s -\s_{0})$
which are $\mathcal{C}^1$ and $\mathcal{C}^2$ in the $Z$-direction respectively. If $\s$ is stable then the index form 
\begin{equation*}
\mathcal{I}(u,v):=\m_{\s}\{|N_h|^{-1}Z(u)Z(v)+quv\}d\s=-\m_{\s} u\mathcal{L}(v)d\s\geq 0
\end{equation*}
 where $\mathcal{L}$ is the following second order differential operator 
 \begin{equation}
\mathcal{L}(v):=|N_{h}|^{-1}\{ Z(Z(v))+|N_{h}|^{-1}g(N,T)(-2|N_{h}|\theta(S)-c_1+2|N_{h}|^2(c_1+g(\tau(Z),\nu_h)))Z(v)-q|N_{h}|v   \},
\end{equation}
with $q$ defined in Theorem \ref{seconvariation}.
\end{proposition}

\begin{proof} Following \cite[Proposition~3.14]{Ri-Ro-Hu} we prove that $\mathcal{L}(v)=div_{\s}(|N_{h}|^{-1}Z(v)Z)+|N_{h}|^{-1}q$. In fact 
\begin{equation*}
div_{\s}(|N_{h}|^{-1}Z(v)Z)=Z(|N_{h}|^{-1}Z(v))+|N_{h}|^{-1}Z(v)div_{\s}Z.
\end{equation*}
So by Lemma \ref{conti}  we have
\begin{equation*}
Z(|N_{h}|^{-1}Z(v))=|N_{h}|^{-1}Z(Z(v))+|N_{h}|^{-2}g(N,T)(-|N_{h}|\theta(S)-c_1+|N_{h}|^2(c_1+g(\tau(Z),\nu_h)))Z(v)
\end{equation*}
and  $div_{\s}(Z)=-g(N,T)\theta(S)+g(N,T)|N_{h}|(c_1+g(\tau(Z),\nu_h))$. 

Finally it is sufficient to observe
\begin{equation*}
\begin{aligned}
0&=\m_{\s}div_{\s}(|N_{h}|^{-1}Z(v)uZ)d\s=\m_{\s}u\,div_{\s}(|N_{h}|^{-1}Z(v)Z)d\s+\m_{\s}|N_{h}|^{-1}Z(v)Z(u)d\s\\
&=\m_{\s}u\mathcal{L}(v)d\s+ \mathcal{I}(u,v).
\end{aligned}
\end{equation*} 

Really we need $u,v\in\mathcal{C}^1_0(\s-\s_0)$, but this condition can be weakened with an approximation argument, as in \cite[Proposition~3.2]{Ri-Ro-Hu}.

\end{proof}

Now we present the analogous of \cite[Lemma~3.17]{Ri-Ro-Hu} and \cite[Lemma~4.1]{Ri-Ro-Hu}, which are a sort of integration by parts and a useful stability operator for non-singular surfaces respectively. The proofs of the following Lemmae are straightforward generalization of the Heisenberg case.

\begin{lemma}\label{intbypart2} Let $\s$ be a $\mathcal{C}^2$ immersed surface in a \phMM, with unit normal vector $N$ and singular set $\s_{0}$. Consider two functions $u\in \mathcal{C}_{0} (\s -\s_{0})$ and $v\in \mathcal{C}(\s -\s_{0})$
which are $\mathcal{C}^1$ and $\mathcal{C}^2$ in the $Z$-direction respectively. Then we have
\begin{equation*}
\m_{\s} |N_{h}| \{  Z(u)Z(v)+uZ(Z(v))+c_1|N_{h}|^{-1}g(N,T)uZ(v)   \}d\s=0.
\end{equation*}

\end{lemma}


\begin{lemma}\label{stabopercanonical} Let $\s$ be a $\mathcal{C}^2$ immersed minimal surface in a \phMM, with unit normal vector $N$ and singular set $\s_{0}$. For any function $u\in \mathcal{C}_{0} (\s -\s_{0})$ 
which is also $\mathcal{C}^1$ in the $Z$-direction we have
\begin{equation*}
\mathcal{I}(u|N_{h}|,u|N_{h}|)=\m_{\s} |N_{h}| \{  Z(u)^2-\mathcal{L}(|N_{h}|)u^2   \}d\s.
\end{equation*}

\end{lemma}


Now it is interesting compute $\mc{L}(|N_h|)$.

\begin{lemma} \label{LN_{h}} Let $\s$ be a $\mathcal{C}^2$ immersed minimal surface in a \phMM. Then
\begin{equation*}
\begin{aligned}
\mathcal{L}(|N_{h}|)&=W+c_1g(\tau(Z),\nu_h)-2c_1|N_{h}|^{-2}(|N_h|\theta(S)-|N_h|^2g(\tau(Z),\nu_h))-c_1^2|N_h|^{-2}g(N,T)^2).
\end{aligned}
\end{equation*}

\end{lemma}

\begin{proof}(of Lemma \ref{LN_{h}}) By (v) in Lemma \ref{conti} we have that
\begin{equation*}
\begin{aligned}
Z(|N_{h}|)=g(N,T)(|N_{h}|\theta(S)+c_1-|N_{h}|^2(c_1+g(\tau(Z),\nu_h)))
\end{aligned}
\end{equation*}
and so
\begin{equation*}
\begin{aligned}
Z(Z(|N_{h}|))=&-|N_{h}|(|N_{h}|\theta(S)+c_1-|N_{h}|^2(c_1+g(\tau(Z),\nu_h)))^2\\
&+Z(|N_{h}|)(g(N,T)\theta(S)-2g(N,T)|N_{h}|(c_1+g(\tau(Z),\nu_h)))\\
&+g(N,T)|N_{h}|Z(\theta(S))-g(N,T)|N_{h}|^2Z(g(\tau(Z),\nu_h)).
\end{aligned}
\end{equation*}
Now we observe that
\begin{equation*}
\begin{aligned}
-g(N,T)(c_1|N_{h}|^{-1}+\theta(S))&Z(|N_{h}|)=Z(|N_{h}|)(g(N,T)\theta(S)-g(N,T)|N_{h}|(c_1+g(\tau(Z),\nu_h)))\\
&+\frac{g(N,T)}{|N_{h}|}(-2|N_{h}|\theta(S)-c_1+|N_{h}|^2(c_1+g(\tau(Z),\nu_h)))Z(|N_{h}|)
\end{aligned}
\end{equation*}
and
\begin{equation*}
\begin{aligned}
-|N_{h}|(|N_{h}|\theta(S)+c_1-|N_{h}|^2(c_1+&g(\tau(Z),\nu_h)))^2=-|N_{h}|^3(\theta(S)-|N_{h}|(c_1+g(\tau(Z),\nu_h)))^2\\
&-c_1^2|N_{h}|-2c_1|N_{h}|^2(\theta(S)-|N_{h}|(c_1+g(\tau(Z),\nu_h))).
\end{aligned}
\end{equation*}
Now it is sufficient to substitute, use  Lemma \ref{conti} (v), and \eqref{deromega} to obtain the required formula.

\end{proof}

\begin{remark} By Lemma \ref{LN_{h}} it is simple to prove that our formula coincide with the one in \cite{CJHMY} in the special case of $\mathcal{C}^3$ surfaces, as the authors obtained that formula by deriving the mean curvature. By
\begin{equation*}
Z\left(\frac{g(N,T)}{|N_{h}|}\right)=|N_{h}|^{-3}Z(g(N,T))
\end{equation*}
and (v) in Lemma \ref{conti} we have
\begin{equation}\label{varyang}
\mathcal{L}(|N_{h}|)=W-c_1g(\tau(Z),\nu_h) +2c_1Z\Big(\frac{g(N,T)}{|N_{h}|}\Big)+c_1^2\frac{g(N,T)^2}{|N_{h}|^2}.
\end{equation}
Equation \eqref{varyang} gives an easy criterion for the stability of \emph{vertical surfaces}, which are the surfaces in which $g(N,T)\equiv 0$ holds. In the Heisenberg group these surfaces are vertical planes and their stability was first proved in \cite{Da-Ga-Nh-Pa}. 
\end{remark}

We conclude this section pasting the variations in the regular and in the singular set, to obtain a stability operator in the spirit of \cite[Proposition~4.11]{Ri-Ro-Hu}. By a tubular neighborhood of $(\s_0)_c\cap supp(u)$ we mean the union of the tubular neighborhood of each singular curve and each singular point lie in $supp(u)$. We note that we are interested in a finite number of singular curves and singular points, as we use variation function $u$ compactly supported and by Theorem \ref{singularsetth}.

\begin{theorem}\label{Stability operator II} Let $\s$ be a $\mathcal{C}^2$ oriented minimal surface immersed in a \phMM, with singular set $\s_0$ and $\partial\s=\emptyset$. If $\s$ is stable then, for any function $u\in\mathcal{C}^1_0(\s)$ such that $Z(u)=0$ in a tubular neighborhood of a singular curve and constant in a tubular neighborhood of an isolated singular point, we have $\mathcal{Q}(u)\geq 0$, where
\begin{equation*}
\mathcal{Q}(u):=\m_{\s}\{|N_h|^{-1}Z(u)^2+qu^2\}d\s+2\m_{(\s_0)_c}(\xi+\zeta)g(Z,\nu)u^2d(\s_0)_c+\m_{(\s_0)_c} S(u)^2  d(\s_0)_c.
\end{equation*}
Here $d(\s_0)_c$ is the Riemannian length measure on $(\s_0)_c$, $\nu$ is the external unit normal to $(\s_0)_c$ and $q,\xi,\zeta$ are defined in Theorem \ref{seconvariation}.
\end{theorem}

\begin{proof} First we observe that $\mathcal{Q}(u)$ is well defined for any $u\in\mathcal{C}_0(\s)$, which is piecewise $\mathcal{C}^1$ in the $Z$-direction and $\mathcal{C}^1$ when restricted to $\s_0$. First we prove
\begin{equation}\label{condizionepreliminare}
\begin{aligned}
\mathcal{Q}(v)\geq 0, &\text{ for any } v\in\mathcal{C}^1_0(\s) \text{ such that } Z(v/g(N,T))=0 \text{ in a small tubular} \\
&\text{ neighborhood } E \text{ of } (\s_0)_c.
\end{aligned}
\end{equation}
Here we denote $\s_0\cap supp(u)$ by $(\s_0)_c$.
Clearly the last hypothesis implies $|N_h|^{-1}Z(u)^2\in\mathcal{L}^1(\s)$. Denoting by $\sigma_0$ the radius of $E$ and by $K$ the support of $v$, respectively, we let $E_{\sigma}$ be the tubular neighborhood of $(\s_0)_c$ of radius $\sigma\in(0,\sigma_0/2)$ and let $h_{\sigma},g_{\sigma}$ be $\mathcal{C}^{\infty}_0(\s)$ functions such that $g_{\sigma}=1$ on $K\cap \overline{E_\sigma}$, $supp(g_{\sigma})\subset E_{2\sigma}$ and $h_{\sigma}+g_{\sigma}=1$ on $K$. Finally we define 
\begin{equation}
U_{\sigma}=(h_{\sigma}v)N+g_{\sigma}\frac{v}{g(N,T)}T.
\end{equation}
Observe that $supp(U_\sigma)\subset K$ and $g(U_{\sigma},N)=v$ on $K$. Now we define a variation $\varphi^{\sigma}_r (p)=exp_p(r(U_{\sigma})_p)$ and the area functional $A_{\sigma}(r)=A(\varphi^{\sigma}_r(\s))$. As this variation is vertical when restricted to $E_{\sigma}$ we have that $A''(\varphi^{\sigma}_r(E_{\sigma}))$ is given by 
\begin{equation*}
\begin{aligned}
A''(\varphi_{0}^{\sigma}(E_{\sigma}))&=\m_{E_{\sigma}}\{2 v^{2}|N_{h}|(g(\tau(Z),\nu_h)^{2}+g(\tau(Z),Z)^{2}   ) \}d\s+\m_{(\s_{0})_{c}}S(v)^{2}d\s_{0}
\end{aligned}
\end{equation*}
and by Theorem \ref{seconvariation} we have
\begin{equation*}
A''(\varphi_{0}^{\sigma}(\s-E_{\sigma}))=\m\limits_{\Sigma-E_{\sigma}} \{|N_{h}|^{-1} Z(u)^2 +u^2 q \}  d\Sigma+\m\limits_{\Sigma-E_{\sigma}} div_{\Sigma}(\xi Z+\zeta Z+\eta S)d\Sigma.
\end{equation*}
If $\s$ is stable then $A''(0)\geq 0$, so using the Riemannian divergence theorem we have
\begin{equation*}
\begin{aligned}
&\m_{E_{\sigma}}\{2 v^{2}|N_{h}|(g(\tau(Z),\nu_h)^{2}+(g(\tau(Z),Z))^{2}   ) \}d\s+\m_{(\s_{0})_{c}}S(v)^{2}d\s_{0}\\
&+\m\limits_{\Sigma-E_{\sigma}} \{|N_{h}|^{-1} Z(u)^2 +u^2 q \}  d\Sigma+2\m\limits_{\partial E_{\sigma}} (\xi Z+\zeta Z)g(Z,\nu)dl\geq 0,
\end{aligned}
\end{equation*}
where $\nu$ is the unit normal pointing into $E_{\sigma}$ and $dl$ denote the Riemannian length element. Letting $\sigma\rightarrow 0$, by the dominated convergence theorem we have proved condition \ref{condizionepreliminare}.

Now we suppose $u\in\mathcal{C}^1_0(\s)$ with $Z(u)=0$ in a tubular neighborhood $E$ of $(\s_0)_c$. Then for any $\sigma\in(0,1)$ let $D_{\sigma}$ be the open neighborhood of $(\s_0)_c$ such that $|g(N,T)|=1-\sigma$ on $\partial D_{\sigma}$. Exists $\sigma_0>0$ such that $D_{\sigma}\subset E$ for $\sigma\in(0,\sigma_0)$. Now we define the function $\phi_{\sigma}:\s\rightarrow [0,1]$ given by 
\begin{displaymath}
\phi_{\sigma}=\begin{cases}|g(N,T)|,    &  \text{in   }    \overline{D_{\sigma}},\\
1-\sigma,   &   \text{in   }  \s-D_{\sigma}.
\end{cases}
\end{displaymath}
We note that $\phi_{\sigma}$ is continuous, piecewise $\mathcal{C}^1$ in the $Z$-direction and the sequence $\{\phi_{\sigma}\}_{\sigma\in(0,\sigma_0)}$ pointwise converge to 1 when $\sigma\rightarrow 0$. Using Lemma \ref{conti}  we have that $|N_h|^{-1}Z(g(N,T))^2$ extends to a continuous function on $\s$, so
\begin{equation*}
\lim\limits_{\sigma\rightarrow\sigma_0}\m_{\s} |N_h|^{-1}Z(\phi_{\sigma})^2d\s=0.
\end{equation*}
Now little modifying $\phi_{\sigma}$ around $\partial D_{\sigma}$ we can consider a sequence of $\mathcal{C}^1$ functions $\{\psi_{\sigma}\}_{\sigma\in(0,\sigma_0)}$ with the same properties. Defining $v_{\sigma}=\psi_{\sigma}u$ we have $\mathcal{Q}(v_{\sigma})\geq 0$ for any $\sigma\in(0,\sigma_0)$ by condition \ref{condizionepreliminare}. Now is sufficient use the dominated convergence theorem and the Cauchy-Schwartz inequality in $L^2(\s)$ to show $\mathcal{Q}(v_{\sigma})\rightarrow \mathcal{Q}(u)$ for $\sigma\rightarrow 0$ and prove the statement.

\end{proof}

\section{Stable minimal surfaces inside a three-dimensional pseudo-hermitian sub-Riemannian manifolds.}

We present a generalization of \cite[Proposition~6.2]{Ro} in the case of a minimal vertical surface of class $\mathcal{C}^2$ inside a three dimensional pseudo-hermitian manifold. A surface $\s$  with unit normal vector $N$ is a \emph{vertical surface} if $g(N,T)\equiv 0$.  Obviously a vertical surface has empty singular set.

\begin{proposition}\label{stability:verticalsurfaces}
Let $\s$ be a $\mathcal{C}^2$ vertical minimal surface inside a \phMM. 
\begin{itemize}
\item [(i)] If $W-c_1g(\tau(Z),\nu_h)>0$ on $\s$, then $\s$ is unstable. 
\item [(ii)] If $W-c_1g(\tau(Z),\nu_h)\leq0$ on $\s$, then $\s$ is stable. 
\end{itemize}
\end{proposition}

\begin{proof} For vertical surfaces \ref{varyang} becomes
\[
\mathcal{I}(u|N_{h}|,u|N_{h}|)=\m_{\s} |N_{h}| \Big\{  Z(u)^2-(W-c_1g(\tau(Z),\nu_h) )u^2   \Big\}d\s.
\]
When $W-c_1g(\tau(Z),\nu_h)>0$ and $\s$ is compact we can use the function $u\equiv 1$ to get the instability. In the non-compact case we can prove (i) with a suitable cut off of the constant function $1$. Point (ii) is immediate.
\end{proof}

It is remarkable that the sign of the quantities $W-c_1g(\tau(Z),\nu_h)$ can be studied at least for three-dimensional Lie groups carrying out a pseudo-hermitian structure. We have the following classification result \cite[Theorem~3.1]{Pe}

\begin{proposition}\label{classification} Let  $M$ be a simply connected contact 3-manifolds, homogeneous in the sense of Boothby and Wang, \cite{MR0112160}. Then $M$ is one of the following Lie group:
	\begin{itemize}
	\item [(1)] if $M$ is unimodular 
	         \begin{itemize}
	         \item the first Heisenberg group $\mathbb{H}^1$ when $W=|\tau|=0$;
	         \item the three-sphere group $SU(2)$ when $W> 2|\tau|$;
	         \item the group $\widetilde{SL(2,\mathbb{R})}$ when $-2|\tau|\neq W<2|\tau|$;
	         \item the group $\widetilde{E(2)}$, universal cover of the group of rigid motions of the Euclidean plane, when $W=2|\tau|>0$;
	         \item the group $E(1,1)$ of rigid motus of Minkowski 2-space, when $W=-2|\tau|<0$;
	         \end{itemize}
	\item [(2)] if $M$ is non-unimodular, the Lie algebra is given by
	        \[
	        [X,Y]=\alpha Y+2T ,\quad  [X,T]=\gamma Y, \quad  [Y,T]=0, \quad \alpha\neq 0,
	        \]
	        where $\{X,Y\}$ is an orthonormal basis of $\hhh$, $J(X)=Y$ and $T$ is the Reeb vector field. In this case $W<2|\tau|$ and when $\gamma=0$ the structure is Sasakian and $W=-\alpha^2$.
	\end{itemize}
	Here $|\tau|$ denote the norm of the matrix of the pseudo-hermitian torsion with respect to an orthonormal basis.
\end{proposition}

A Lie Group is \emph{unimodular} when his left invariant Haar measure is also right invariant \cite[p.~248]{Pe}.

We remark that in \cite{Pe} the author gives the classification in terms of the equivalent invariant $W_1=W/4$ and $|\tau_1|=2\sqrt{2}|\tau|$. It is simple to show that if $M$ is unimodular then
\begin{equation}\label{Wunimodular}
W=\frac{c_1(c_3-c_2)}{2} \quad \text{and} \quad |\tau|=\frac{|c_2+c_3|}{2}, 
\end{equation}
where the Lie algebra of  $M$ is defined by
\[
[X,Y]=c_1 T, \quad [X,T]=c_2 Y, \quad [Y,T]=c_3 X,
\]
with $\{X,Y\}$ orthonormal basis of $\hhh$, $J(X)=Y$, $T$ the Reeb vector field and the normalization $c_1=-2$. In the non-unimodular case we have
\begin{equation}\label{Wnonunimodular}
W=-\alpha^2-\gamma \quad \text{and} \quad |\tau|=|\gamma|.
\end{equation}
Furthermore in a unimodular sub-Riemannian Lie group $G$ the matrix of $\tau$ in the ${X,Y,T}$ basis is 
\[
\left( \begin{array}{ccc} 
0 & \frac{c_2+c_3}{2}& 0 \\
\frac{c_2+c_3}{2} & 0& 0\\
0 & 0 & 0 \end{array} \right)
\]
and by \cite[p.~38]{Dr-To} we can compute the following derivatives
\begin{equation}\label{Gchistoffel}
\begin{aligned}
&\n_{X}X=0, \quad  \n_{Y}X=0,  \quad    \n_{T}X=\frac{c_3-c_2}{2}Y,    \\
&  \n_{X}Y=0, \quad    \n_{Y}Y=0, \quad    \n_{T}Y=\frac{c_2-c_3}{2}X.
\end{aligned}
\end{equation}


If we consider another orthonormal basis $\{X_{1},Y_{1},T\}$ where $J(X_{1})=Y_{1}, X_{1}=a_{1}X+a_{2}Y, Y_{1}=-a_{2}X+a_{1}Y$ the new torsion matrix becomes 
\begin{equation}\label{torsionframe}
\left( \begin{array}{ccc} 
(c_{2}+c_{3})a_{1}a_{2} & \frac{c_{2}+c_{3}}{2}(a_{1}^2-a_{2}^2)& 0 \\
\frac{c_{2}+c_{3}}{2}(a_{1}^2-a_{2}^2)& (c_{2}+c_{3})a_{1}a_{2}& 0\\
0 & 0 & 0 \end{array} \right).
\end{equation}

\begin{lemma}\label{R-e1tor} Let $\s$ a  surface of constant mean curvature $H$ immersed in a unimodular Lie group $G$ . Then
\[
g(R(T,Z)\nu_h,Z)-Z(g(\tau(Z),\nu_h))=2Hg(\tau(Z),Z),
\]
which vanishes when $\s$ is minimal. 
\end{lemma}

\begin{proof} By \eqref{torsionframe} we can express $g(\tau(Z),\nu_h)=(c_2+c_3)(1-2g(Z,X)^2)/2$ and
\[
Z(g(\tau(Z),\nu_h))=-2(c_2+c_3)g(Z,X)(g(\n_{Z}Z,X)+g(\n_{Z}X,Z)). 
\]
Taking into account \eqref{eq:carcurvesph} and \eqref{Gchistoffel} we obtain
\begin{equation}\label{dertor}
Z(g(\tau(Z),\nu_h))=2 H g(\tau(\nu_h),\nu_h).
\end{equation}
On the other hand $\nu_h(g(\tau(Z),Z))=(c_2+c_3)\nu_h(g(\nu_h,X)g(\nu_h,Y))$, calculating
\begin{equation}\label{dertornu}
\begin{aligned}
\nu_h(g(\tau(Z),Z))=&(c_2+c_3)g(\nu_h,Y)(g(\n_{\nu_h}\nu_h,X)+g(\n_{\nu_h}X,\nu_h))\\
&+(c_2+c_3)g(\nu_h,X)(g(\n_{\nu_h}\nu_h,Y)+g(\n_{\nu_h}Y,\nu_h))\\
&=-2\theta(\nu_h)g(\tau(Z),\nu_h),
\end{aligned}
\end{equation}
where we have used \eqref{Gchistoffel}. Finally taking into account \eqref{equivRTRLCT}, \eqref{curvaturaT}, \eqref{dertor} and \eqref{dertornu} we get the claim. 
\end{proof}

\begin{lemma}\label{indiceuv} Let $\s$ be a $\mathcal{C}^2$ immersed minimal surface in $M$. Consider two functions $u\in\mathcal{C}(\s-\s_0)$ and $v\in\mathcal{C}(\s-\s_0)$ which are $\mathcal{C}^1$ and $\mc{C}^2$ in the $Z$-direction, respectively. If $v$ never vanishes, then
\begin{equation}\label{eq:indiceuv}
\begin{aligned}
\mc{I}(uv^{-1}|N_h|,uv_{-1}|N_h|)=&\m_{\s}|N_h|v^{-2}Z(u)^2d\s\\
&+\m_{\s} |N_h|u^2\{  Z(v^{-1})^2-\frac{1}{2}Z(Z(v^{-4}))-\frac{c_1}{2}\frac{g(N,T)}{|N_h|}Z(v^{-4})  \}d\s\\
&-\m_{\s} |N_h|\mc{L}(|N_h|)(uv^{-1})^2d\s.
\end{aligned}
\end{equation}

\end{lemma}

The proof is the same as of \cite[Lemma~4.3]{Ri-Ro-Hu} except that Lemma \ref{intbypart2} is used instead of \cite[Lemma~3.17]{Ri-Ro-Hu}.

\begin{proposition}\label{segnoL} Let $\s$ be a complete orientable $\mathcal{C}^2$ minimal surface  with empty singular set immersed in a \phMM. We suppose that $g(R(T,Z)\nu_h,Z)-Z(g(\tau(Z),\nu_h))=0$ on $\s$ . If 
\[
(W-c_1g(\tau(Z),\nu_h))(p_0)\geq 0
\] 
for some $p_0\in \s$, then the operator $\mc{L}$ satisfies $\mc{L}(|N_h|)\geq 0$ on the characteristic curve $\ga_0$ passing through $p_0$. Moreover, $\mc{L}(|N_h|)=0$ over $\s$ if and only if $g(N,T)=0$ and $W-c_1g(\tau(Z),\nu_h)=0$ on $\ga_0$.

\end{proposition}

\begin{proof} We consider a point $p\in\s$. Let $I$ an open interval containing the origin and $\alpha:I\rightarrow \s$ a piece of the integral curve of $S$ passing through $p$. Consider the characteristic curve $\gamma_\eps(s)$ of $\s$ with $\gamma_\eps(0)=\alpha(\eps)$. We define the map $F:I\times\rr\rightarrow \s$ given by $F(\eps,s)=\gamma_\eps(s)$ and denote $V(s):=(\partial F/\partial \eps)(0,s)$ which is a Jacobi-like vector field along $\gamma_0$, Proposition \ref{jacobifields}. Clearly $V(0)=(S)_p$. We denote by $'$ the derivatives of functions depending on $s$, and the covariant derivative along $\gamma_0$ respect to $\n$ and $\dot{\gamma}_0$ by $Z$. By \eqref{vt1} and \eqref{vt2}
\begin{equation}\label{vt1s}
g(V,T)'=-c_1g(V,\nu_h),
\end{equation}
\begin{equation}\label{vt2s}
\frac{1}{c_1}g(V,T)''=c_1\lambda g(V,Z)-Z(g(V,\nu_h))
\end{equation}
since 
\[
g(V',\nu_h)=Z(g(V,\nu_h))+g(V,J(\n_{Z}Z))=Z(g(V,\nu_h)).
\]
Now we show that $\{V,Z\}$ is a basis of $T\s$ along $\gamma_0$. It is sufficient to show that $g(V,T)$ and $g(V,\nu_h)$ do not vanish simultaneously. Suppose there exists $s_0$ such that $g(V,T)(s_0)=g(V,\nu_h)(s_0)=0$. This means that $V(s_0)$ is co-linear with $(Z)_{s_0}$ and $$g(V,T)'(s_0)=g(V,T)''(s_0)=0$$ by \eqref{vt1s} and \eqref{vt2s}.
As $g(V,T)$ satisfies the differential equation in Proposition \ref{jacobifields} (iv) we deduce $g(V,T)=0$ along $\gamma_0$ which is impossible as $g(V,T)(0)=-|N_h|<0$. We have proved that $g(V,T)$ never vanishes along $\gamma_0$ as $\s_0$ is empty.

By \eqref{dertor} we have $W-c_1g(\tau(Z),\nu_h)=k^2$, with $k\geq0$. If $k=0$ then solving the ordinary differential equation in Proposition \ref{jacobifields} (iv) we have
\[
g(V,T)(s)=as^2+bs+c,
\]
where $a,b,c$ are given by
\begin{align*}
&a=g(V,T)''(0)/2=-c_1Z(g(N,T))/2,\\
&b=g(V,T)'(0)=-c_1g(N,T),\\
&c=g(V,T)(0)=-|N_h|.
\end{align*}
Now $g(V,T)\neq 0$ implies $b^2-4ac<0$ or $a=b=0$. In the first case we get
\[
b^2-4ac=\{c_1^2g(N,T)^2-2c_1|N_h|Z(g(N,T))\}\geq -\{c_1^2g(N,T)^2+2c_1|N_h|Z(g(N,T))\}
\]
and the right term is equal to
\[
-|N_h|^2\left\{2c_1Z\left(\frac{g(N,T)}{|N_h|}\right)+c_1^2\frac{g(N,T)^2}{|N_h|^2}\right\},
\]
which implies $\mc{L}(|N_h|)\geq 0$. On the other hand $a=b=0$ implies that $\s$ is a vertical surface and $\mc{L}(|N_h|)=0$. We note that in any vertical surface $b^2-4ac=0$ so that $\mc{L}(|N_h|)=0$.

Now we suppose $k\neq0$. Then by Proposition \ref{jacobifields} (iv) we get 
\[
g(V,T)(s)=\frac{1}{k}(a\sin(k s)-b\cos(k s))+c, 
\]
with $a,b,c$ given by
\begin{align*}
&a=g(V,T)'(0)=-c_1g(N,T),\\
&b=\frac{1}{k}g(V,T)''(0)=-\frac{c_1}{k}Z(g(N,T)),\\
&c=\frac{1}{k^2}g(V,T)''(0)+g(V,T)(0)=\frac{b}{k}-|N_h|.
\end{align*}
As in \cite[Proof of Proposition~6.6]{Ro} we have $g(V,T)(s)\neq 0$ for all $s$ if and only if 
\[
0<k^2|N_h|^2-2 k|N_h|b-a^2=|N_h|^2\mc{L}(|N_h|),
\]
which implies $\mc{L}(|N_h|)>0$.

\end{proof}

\begin{lemma}\label{char:iniettive} Let $\s$ be a $\mc{C}^2$ complete, oriented, immersed, CMC surface with empty singular set in a \phMM. Then any characteristic curve of $\s$ is an injective curve or a close curve. 
\end{lemma}

\begin{proof} Since characteristic curves are the trajectories of the vector field $Z$, they are injectives or closed curves.

\end{proof}

\begin{remark} It is remarkable that Lemma \ref{char:iniettive}, together with \cite[Remark~6.8]{Ro}, implies that \cite[Theorem~6.7]{Ro} holds for all  homogeneous Sasakian sub-Riemannian 3-manifolds. We only have to reason as in the last part of the proof of Proposition \ref{necessariastabilita} below. 
\end{remark}

\begin{proposition}\label{necessariastabilita} Let $\Sigma$ be a $\mathcal{C}^2$ complete orientable surface with empty singular set immersed in a \phMM. We suppose that $g(R(T,Z)\nu_h,Z)-Z(g(\tau(Z),\nu_h))=0$ on $\s$ and the quantity $W-c_1g(\tau(Z),\nu_h)$ is constant along characteristic curves. We also assume that all characteristic curves in $\Sigma$ are either closed or non-closed. If $\Sigma$ is a stable minimal surface, then $W-c_1g(\tau(Z),\nu_h)\leq 0$ on $\s$. Moreover, if $W-c_1g(\tau(Z),\nu_h)=0$ then $\Sigma$ is a stable vertical surface.

\end{proposition}

\begin{proof} We need to prove that when exists $p\in\s$ such that $W-c_1g(\tau(Z),\nu_h)> 0$ in $p$ and $\mc{L}(|N_h|)\neq 0$ over the characteristic curve passing through $p$ in $\s$, then $\s$ is unstable, in virtue of Proposition \ref{segnoL}. We consider $p\in \s$ such that $\mc{L}(|N_h|)(p)>0$. We denote by $\gamma_0(s)$ the characteristic curve passing through $p$ and we denote by $\alpha(\eps)$ the integral curve of $S$ passing through $p$, parametrized by arc-length. As the surface is not singular $\s$ is foliated by characteristic curves, we denote by $\gamma_\eps (s)$ the characteristic curve passing through $\alpha(\eps)$ parametrized by arc-length.
We obtain a $\mc{C}^1$ map $F:I\times I'\rightarrow \s$ given by $F(\eps,s)=\gamma_{\eps}(s)$ which parametrizes a neighborhood of the characteristic curve $\gamma_0$ on $\s$, where $I'$ is an interval, compact or not, where live the parameter $s$ and $I=[-\eps_0,\eps_0]$ with $\eps_0\in\rr$ eventually small. 
By Proposition \ref{jacobifields} $V_\eps(s):=(\partial F/\partial \eps)(\eps,s)$ is a Jacobi-like vector field along $\gamma_\eps$ and the function $g(V_\eps,T)$ never vanishes since $\s_0=\emptyset$. Furthermore $V_\eps(0)=(S)_{\alpha(\eps)}$ implies that $g(V_\eps,T)<0$. We define the function $f_\eps:=g(V_\eps,S)$ and it is immediate that $g(V_\eps,T)=-f_\eps |N_h|$ and $g(V_\eps,\nu_h)=f_\eps g(N,T)$ where $|N_h|$ and $g(N,T)$ are evaluated along $\gamma_\eps$. The Riemannian area element of $\s$ with respect to the coordinates $(\eps,s)$ is given by 
\[
d\s=(|V_\eps|^2-g(V_\eps,\dot{\gamma_\eps}))^{1/2}=f_\eps\, ds\, d\eps.
\]
We define the function
\begin{equation}\label{def:v}
v(\eps,s):=|g(V_\eps,T)(s)|^{1/2}=(f_\eps |N_h|)^{1/2},
\end{equation}
which is positive, continuous on $I\times I'$ and $\mc{C}^{\infty}$ along characteristic curves, by Proposition \ref{jacobifields}. Denoting $v_\eps (s)=v(\eps,s)$ and denoting by $'$ the derivatives with respect to $s$, by \eqref{vt1s} and \eqref{vt2s} we get 
\[
(v_\eps^{-2})'=g(V_\eps,T)^{-2}g(V_\eps,T)'=-c_1\frac{g(N,T)}{f_\eps |N_h|^2},
\]
\[
(v_\eps^{-4})'=-2g(V_\eps,T)^{-3}g(V_\eps,T)'=-2c_1\frac{g(N,T)}{f_\eps^2|N_h|^3},
\]
\begin{equation*}
\begin{aligned}
(v_\eps^{-4})''&=6g(V_\eps,T)^{-4}(g(V_\eps,T)')^2-2g(V_\eps,T)^{-3}g(V_\eps,T)''=4c_1^2\frac{g(N,T)^2}{f_\eps^2|N_h|^4}-2c_1\frac{Z(|N_h|^{-1}g(N,T))}{f_\eps^2|N_h|^2},
\end{aligned}
\end{equation*}
where we have used $g(V_\eps,\nu_h)=-g(V_\eps,T)|N_h|^{-1}g(N,T)$, and consequently 
\begin{equation}
\begin{aligned}
((v_\eps^{-2})')^2-\frac{1}{2}(v_\eps^{-4})''-\frac{c_1}{2}\frac{g(N,T)}{|N_h|}(v_\eps^{-4})'&=c_1\frac{Z(|N_h|^{-1}g(N,T))}{f_\eps^2|N_h|^2}\\
&=\frac{\mc{L}(|N_h|)}{2f_\eps^2|N_h|^2}-\frac{W-c_1\IA+c_1^2|N_h|^{-2}g(N,T)^2}{2f_\eps^2|N_h|^2}.
\end{aligned}
\end{equation}

Now we consider a function $\phi:\rr\rightarrow\rr$ such that $\phi\in\mc{C}^{\infty}_0(I)$ and $\phi(0)>0$. Let $\rho$ a positive constant such that $|\phi'(\eps)|\leq \rho$ for any $\eps\in\rr$. We distinguish two cases. First we suppose that the family of curves $\gamma_\eps$ is defined in the whole real line for $\eps$ small enough. For any $n\in\nn$ we consider the function $u_n:I\times I'\rightarrow\rr$ defined by $u_n(\eps,s):=\phi(\eps)\phi(s/n)$, with $I'=\rr$. At this point we can conclude as in \cite[proof of Theorem~6.7]{Ro}.

In the second case we consider a family of closed curves $\gamma_\eps$ with eventually different length $l_\eps$.  We can parametrize all the curves  as $\gamma_\eps(t):I'\rightarrow\s$, with $t=s\, l_0 /l_\eps$ and $I'=[0,l_0]$. In this case we get
\begin{equation*}
\begin{aligned}
((v_\eps^{-2})')^2-\frac{1}{2}(v_\eps^{-4})''-\frac{c_1}{2}&\frac{g(N,T)}{|N_h|}(v_\eps^{-4})'=\frac{l_0}{l_\eps}c_1\frac{Z(|N_h|^{-1}g(N,T))}{f_\eps^2|N_h|^2}\\
&=\frac{l_o}{l_\eps}\frac{\mc{L}(|N_h|)}{2f_\eps^2|N_h|^2}-\frac{l_0}{l_\eps}\frac{W-c_1\IA+c_1^2|N_h|^{-2}g(N,T)^2}{2f_\eps^2|N_h|^2}.
\end{aligned}
\end{equation*}
Now it is sufficient reasoning as above changing the definition of the function $\phi_n(t):=\phi(0)$ to conclude as in \cite[proof of Theorem~6.7]{Ro}.

We observe that, chosen a point $p\in\s$, the curve $\gamma_0$ passing through $p$ can be closed (resp. non-closed) but the other characteristic curves $\gamma_\eps$ can be non-closed (reps. closed) even for $\eps_0$ small. In this case we can choose our initial point in another non-closed (reps. closed) curves. 

\end{proof}

\begin{remark} The proof of Proposition \ref{necessariastabilita} works under weaker assumptions, i.e. when the closed and non-closed characteristic curves of $\s$  are not dense the ones in the others.
\end{remark}

\begin{corollary}\label{su2} There are not complete stable minimal surfaces with empty singular set in the three-sphere group $SU(2)$.
\end{corollary}

\begin{proof} By Proposition \ref{classification} in $SU(2)$ we have $W-2g(\tau(Z),\nu_h)>0$ and we get the statement using Theorem \ref{necessariastabilita}.
\end{proof}

\begin{remark} In \cite[Corollary~6.9(ii)]{Ro} the author shows that  complete stable minimal surfaces with empty singular set do not exist in the pseudo-hermitian 3-sphere, which is the only Sasakian structure of $SU(2)$.
\end{remark}

\section{Classification of complete, stable, minimal surfaces in the roto-traslation group $\mc{RT}$.}

We consider the group of rigid motions of the Euclidean plane. The underlying manifold is $\mathbb{R}^2\times \mathbb{S}^1$ where the horizontal distribution $\hhh$ is generated by the vector fields
\begin{equation*}
X=\frac{\partial}{\partial \alpha} \,\,\text{ and }\,\, Y=cos (\alpha)\, \frac{\partial}{\partial x}+sin( \alpha)\, \frac{\partial}{\partial y},
\end{equation*}
the Reeb vector field is 
\begin{equation*}
T=sin (\alpha)\, \frac{\partial}{\partial x}-cos( \alpha)\, \frac{\partial}{\partial y}
\end{equation*}
and the contact form is $\omega=sin (\alpha) \,dx - cos( \alpha)\, dy$, \cite{Ca-Da-Pa-Ty}. Furthermore we have the following Lie brackets 
\[
[X,Y]=-T,\,[X,T]=Y,\, [Y,T]=0
\] 
which imply $W=1/2$ and that the matrix of the pseudo-hermitian torsion with respect to the basis  $\{X,Y,T\}$ is
\begin{displaymath}
 \left( \begin{array}{ccc} 
0 & \frac{1}{2}& 0 \\
\frac{1}{2} & 0& 0\\
0 & 0 & 0 \end{array} \right)
\end{displaymath}


By \cite[Theorem~1.2]{Hl-Pa2} a characteristic curve $\gamma(t)=(x(t),y(t),\alpha(t))$ of curvature $\lambda=0$ with initial conditions $\gamma(0)=(x_0,y_0, \alpha_0)$ and $\dot{\gamma(0)}=(\dot{x_0},\dot{y_0},\dot{\alpha_0})$ in $\mc{RT}$  is of the form
\begin{equation}\label{RTcar1}
\ga(t)=(x_{0}+R_{0} \cos (\alpha_{0}) t,\, y_{0}+R_{0} \sin (\alpha_{0}) t,\, \alpha_{0})
\end{equation}
when $\theta_0=0$ or 
\begin{equation}\label{RTcar2}
\ga(t)=(x_{0}+(R_{0}/\dot{\alpha}_{0})(\sin(\alpha(t))-\sin(\alpha_{0}) ),\, y_{0}+(R_{0}/\dot{\alpha}_{0})(\cos(\alpha_{0})-\cos(\alpha(t))),\, \alpha_{0}+\dot{\alpha}_{0} t)
\end{equation}
otherwise,  where $R_{0}=\sqrt{\dot{x_0}^2+\dot{y_0}^2}$. We underline that the first family of curves is composed by sub-Riemannian geodesic but the second one only when $R_{0}=0$.

We investigate the equation of a minimal surface $\s$ defined as the zero level set of a function $u(\alpha,x,y)$. We consider the horizontal unit normal   and the characteristic field
\[\nu_{H}=\frac{(u_{\alpha}X+(\cos(\alpha) u_{x}+\sin(\alpha) u_{y})Y)}{(u_{\alpha}^2+\cos^2(\alpha)u_{x}^2+\sin^2(\alpha)u_{y}^2)^{1/2}} ,\, Z=\frac{(\cos(\alpha) u_{x}+\sin(\alpha) u_{y})X-u_{\alpha}Y}{(u_{\alpha}^2+\cos^2(\alpha)u_{x}^2+\sin^2(\alpha)u_{y}^2)^{1/2}}
\]
respectively. By a direct computation we get the minimal surface equation
\begin{equation}\label{RTminimal}
\begin{aligned}
&u_{\alpha}^2(\cos^2(\alpha)u_{xx}+2\cos(\alpha)\sin(\alpha)u_{xy}+\sin^2(\alpha)u_{yy})+(\cos(\alpha)u_{x}+\sin(\alpha)u_{y})^2 u_{\alpha \alpha}\\
&-u_{\alpha}(\cos(\alpha)u_{x}+\sin(\alpha)u_{y})(2\cos(\alpha)u_{\alpha x}+2\sin(\alpha)u_{\alpha y}-\sin(\alpha)u_{x}+\cos(\alpha)u_{y})=0.
\end{aligned}
\end{equation}

\begin{remark}\label{ossWrt} In $\mc{RT}$ we can express 
\[
g(\tau(Z),Z)=g(Z,X)g(Z,Y)=-g(\nu_h,X)g(\nu_h,Y)
\]
and
\[
g(\tau(Z),\nu_h)=1/2-g(\nu_h,Y)^2
\]
which imply $W-g(\tau(Z),\nu_h)=g(\nu_h,Y)^2=g(Z,X)^2$.
\end{remark}

\begin{corollary} Let $\s$ be a  $\mc{C}^2$ stable, oriented, complete, immersed  minimal surface in $\mc{RT}$ with empty singular set. Then $\s$ is a vertical plane  of the form $\s_{a}=\{(x,y,\alpha)\in\mc{RT}:\alpha=a\in\mc{S}^1\}$ .
\end{corollary}

We note that there exists another family of vertical surfaces composed of the left-handed helicoids  $\s_{b}=\{(x,y,\alpha)\in\mc{RT}: \cos(b\,\alpha)x+\sin(b\,\alpha)y=0, b\in\mc{S}^1\}$, that are unstable minimal surfaces. In fact the horizontal normal of $\s_b$ is 
\[
\nu_h=\frac{(-\sin(\alpha)x+\cos(\alpha)y)X+Y}{(1+(-\sin(\alpha)x+\cos(\alpha)y)^2)^{1/2}}
\]
which implies $W-g(\tau(Z),\nu_h)>0$ outside the line $\{x=y=0\}$.

\begin{lemma}\label{stazionariopuntosingolare} In $\mc{RT}$ there do not exist minimal surfaces with isolated singular points. 
\end{lemma}

\begin{proof} We can suppose that the singular point is the origin. Then $T_0\s =span\{\partial_x,\partial_\alpha\}$.  The unique way to construct a minimal surface is to to put together all characteristic curves starting from $0$, in the directions of $T_0\s$ with curvature $\lambda=0$, Theorem \ref{singularsetth}. But in this way we construct a right-handed helicoid denoted $\s_c$ below, which contains a singular line.
\end{proof}

\begin{lemma}\label{stazionariaconcurva} Let $\s$ be a complete area-stationary surface of class $\mc{C}^2$ in $\mc{RT}$ which contains a singular curve $\Ga$. Then $\s$ is a right-handed helicoid $\s_c$ or a plane $\s_{a,b,c}$ defined below.
\end{lemma}

\begin{proof} We consider a singular curve $\Ga(\eps)$ in $\s$. Then as $\s$ is foliated by characteristic curves we can parametrize it by the map $F(\eps,s)=\ga_\eps(s)$, where $\ga_\eps(s)$ is the characteristic curves with initial data $\ga_\eps(0)=\Ga(\eps)$ and $\dot{\ga}_\eps(0)=J(\dot{\Ga}(\eps))$. We define the function $V_\eps(s):=(\partial F/\partial \eps)(s,\eps)$ that is a smooth Jacobi-like vector field along $\ga_\eps(s)$. The vertical component of $V_\eps$ satisfies the ordinary differential equation
\[
g(V_\eps,T)'''+ k_\eps \,g(V_\eps,T)'=0,
\]
with $k_\eps=g(\dot{\ga}_\eps(s),X)^2$ that is constant along $\ga_\eps(s)$. We suppose that a characteristic curve $\ga_{\tilde{\eps}}(s)$ is not a sub-Riemannian geodesic, it means that $0<k_{\tilde{\eps}}<1$. As $g(V_\eps,T)'(0)=0$ and $g(V_\eps,T)''(0)=0$ by \eqref{vt1}, \eqref{vt2} and the fact that $\Ga$ is a singular curve, we get 
\begin{equation}\label{vtRT}
g(V_\eps,T)(s)=-\frac{1}{\sqrt{k_\eps}}\sin(\sqrt{k_\eps}\, s)
\end{equation} 
and we find another singular point at distance $\pi/\sqrt{k_\eps}$. The singular point is contained in a singular curve $\Ga_1$ composed of points of the type $\ga_\eps(s_\eps)$, with $s_\eps=\pi/\sqrt{k_\eps}$. $\s$ area-stationary implies $g(\dot{\Ga}_1(\eps),J(\dot{\ga}_\eps(s_\eps)))=0$. Now we prove that $g(V_\eps,\dot{\ga}_\eps)(s)$ is constant along $\ga_\eps$.  It is zero in the initial point and we suppose it is increasing or decreasing. By point (ii) in Proposition \ref{jacobifields} we get that it has a maximum or a minimum in $s_\eps$ and so $V_\eps(s_\eps)$ and $\dot{\ga}_\eps(s_\eps)$ are co-linear. This is impossible and we have proved $V_\eps(s_\eps)=\dot{\Ga}_1(\eps)$. Finally integrating $g(V_\eps,\dot{\ga}_\eps)(s)$ along $\ga_\eps$ by point (ii) in Proposition  \ref{jacobifields} we get 
\[
0=\m_{0}^{s_\eps}g(V_\eps,\dot{\ga}_\eps)'(s)ds=-\m_{0}^{s_\eps}g(V_\eps,T)(s)g(\tau(\dot{\ga}_\eps),\dot{\ga}_\eps)(s)ds,
\]
that is impossible since $g(V_\eps,T)>0$ on $(0,s_\eps)$ and $g(\tau(\dot{\ga}_\eps),\dot{\ga}_\eps)=g(\dot{\ga}_\eps,X)\sqrt{1-g(\dot{\ga}_\eps,X)^2}$ is a constant different from zero. We have proved that each $\ga_\eps$ is a sub-Riemannian geodesic and $k=k_\eps$ is equal to 0 or 1. When $k=0$ we get the surface a right-hand helicoid and when $k=1$ we get a plane.

\end{proof}

\begin{remark} In \cite[Example~2.1]{Na} the author gives examples of minimal surfaces of equations $ax+b\sin(\alpha)+c=0$ and $x-y+c(\sin(\alpha+\cos(\alpha)))+d=0$. Also the surfaces $ay-b\cos(\alpha)+c=0$ and $x+y+c(\sin(\alpha+\cos(\alpha)))+d=0$ are minimal surfaces with a similar property, in fact they satisfy $g(\tau(Z),\nu_h)=0$ . We remark that all these examples are not area-stationary. 

For example in the surface  described by $x+\sin(\alpha)=0$ we have $Z=(-\cos(\alpha)X+\cos(\alpha)Y/(2|\cos(\alpha)|)$ that is not orthogonal to the singular curves $\Ga_1=\{(-1,y,\pi/2)\in\mc{RT}:y\in\rr\}$ and $\Ga_2=\{(1,y,3\pi/2)\in\mc{RT}:y\in\rr\}$.
\end{remark}

\begin{lemma}\label{curvamin} Let $\s$ be a surface defined by a function $u(x,y)=0$, with $u:\mathbb{R}^2\rightarrow \mathbb{R}$ of class $\mathcal{C}^2$ and $(u_{x},u_{y})\neq (0,0)$. Then $\s$ is a minimal surface that  is area-stationary if and only if it is a plane $\s_{a,b,c}=\{(x,y,\alpha)\in\mc{RT}:ax+by+c=0, \,a,b\in\rr,\, c\in\mc{S}^1\}$. 
\end{lemma}

\begin{proof} It is sufficient observe that $u_{\alpha}$ or $u_{\alpha \alpha}$ multiply each term of equation (\ref{RTminimal}). Furthermore it is clear that a surface $\s$ of the type $u=u(x,y)$ contains two singular curves whose union is $\s_0=\{(x,y,\alpha)\mc{RT}:  \cos(\alpha)u_x+\sin(\alpha)u_y=0  \}$; by Lemma \ref{stazionariaconcurva} the surface is a plane $\s_{a,b,c}=\{(x,y,\alpha)\in\mc{RT}:ax+by+c=0, \,a,b\in\rr,\, c\in\mc{S}^1\}$ .
\end{proof}

In the sequel we investigate the stability of the two families of area-stationary surfaces that contains singular curves. 

\begin{proposition} All planes $\s_{a,b,c}=\{(x,y,\alpha)\in\mc{RT}:ax+by+c=0, \,a,b\in\rr,\, c\in\mc{S}^1\}$ are unstable area-stationary surfaces. 
\end{proposition}

\begin{proof} We take for simplicity a plane of equation  $y=0$. Then we have 
\[
\nu_h=\frac{\sin(\alpha)}{|\sin(\alpha)|}Y\quad Z=\frac{\sin(\alpha)}{|\sin(\alpha)|}X.
\]
Then we get $g(\tau(Z),\nu_h)=-1/2$ and $W-g(\tau(Z),\nu_h)=1$ by Remark \ref{ossWrt}. Furthermore using \eqref{Gchistoffel} we can compute $\theta(S)=-|N_h|$ and putting a function $u=u(x)$, with $u\in\mc{C}^\infty_0([-x_0,x_0])$ and $x_0>0$, in the stability operator in Theorem \ref{Stability operator II} we get
\[
\mc{Q}(u)=\left(\m_{[-x_0,x_0]}u(x)^2\,dx\right)\left(-\m_{[0,2\pi]}\frac{1}{4}|\sin(\alpha)|^3|\cos(\alpha)|\,d\alpha \right)+2\m_{[-x_0,x_0]}u'(x)^2\,dx
\]
and as 
\[
\inf\left\{  \left(   \m_{\rr}u'(x)\,dx \right)\left(\m_{\rr}u(x)^2\,dx\right)^{-1}:u\in\mc{C}^{\infty}_0(\rr)  \right\}=0,
\]
there exists a function $u\in\mc{C}^\infty_0([-x_0,x_0])$ such that $\mc{Q}(u)<0$.

\end{proof}

\begin{remark} A plane characterized by equation $ax+by+c\alpha=d$ is not minimal if $a,b,c\neq 0$. 
\end{remark}

\begin{proof} That plane is minimal if and only if the following equation hold: 
\begin{equation*}
c\{ ab(\cos^2\alpha-\sin^2\alpha)+\cos\alpha \sin\alpha (b^2-a^2)  \}=0,
\end{equation*}
that implies $c=0$ or $a=b=0$. 
 \end{proof}

\begin{proposition} Let $\s_{c}=\{(x,y,\alpha)\in\mc{RT}:x\sin(c\,\alpha)-y\cos(c\,\alpha)=0, c\in\mc{S}^1\}$. Then $\s_c$ is a stable, area-stationary surface. 
\end{proposition}

\begin{proof} By a direct substitution in \eqref{RTminimal} $\s_c$ is minimal.  Now we suppose $c=1$ for simplicity and we have 
\[
\nu_h=\frac{x\cos(\alpha)+y\sin(\alpha)}{|x\cos(\alpha)+y\sin(\alpha)|}X, \, \,\,\,\,Z=\frac{x\cos(\alpha)+y\sin(\alpha)}{|x\cos(\alpha)+y\sin(\alpha)|}Y
\]
 outside the only singular curve $\Ga_0=\{(x,y,\alpha)\in\mc{S}^1:x=y=0\}$, so the characteristic curves meet orthogonally the singular one.  
 
 Now by \eqref{Gchistoffel} we have $\theta(S)=|N_h|$ and by Remark \ref{ossWrt} we get $-W+g(\tau(Z),\nu_h))=0$ and $g(\tau(Z),\nu_h)=1/2$. Then the stability operator for non-singular surfaces in Theorem \ref{Stability operator II} become 
\[
\mc{Q}(u)=\m_{\s} \left\{|N_h|^{-1}Z(u)^2+|N_h|\left(1-\frac{1}{4}|N_h|^2\right)u^2\right\}d\s +4\m_{\Ga_0}(u\big|_{\Ga_0})^2d\Ga_0+\m_{\Ga_0}S(u\big|_{\Ga_0})^2d\Ga_0,
\]
which is non-negative for all functions $u\in\mc{C}^1_0(\s_c)$. 
\end{proof}

\begin{theorem}\label{classificationstablesRT} Let $\s$ be a stable, immersed, oriented and complete surface of class $\mc{C}^2$ in $\mc{RT}$. Then we distinguish two cases:
\begin{itemize}
\item [(i)] if $\s$ is a non-singular surface, then it is a vertical plane $\s_a$;
\item [(ii)] if $\s$ is a surface with non-empty singular set, then it is the right-handed helicoid $\s_c$. 
\end{itemize}
\end{theorem}

Finally we would remark that the family of planes $\s_a$ are area-minimizing by a standard calibration argument, in fact they form a family of area-stationary surfaces who foliate $RT$.

\section*{Acknowledgment}

It is a duty as well as a pleasure thanks Manuel Ritor\'e for many fruitful discussions, suggestions and a carefully reading of a first version of the paper. We also thank C\'esar Rosales for some discussions about his recent paper \cite{Ro} and helpful comments. Finally we are grateful to Francesco Serra Cassano who raised our attention to Proposition 1.20 in \cite{Fo-St}.

\bibliography{variazionicopia}

\end{document}